\documentclass[11pt]{article}
\usepackage{geometry}
\usepackage{amsmath,amssymb,amsthm}
\usepackage{charter}
\usepackage{cite}
\usepackage{xcolor}
\usepackage[nottoc]{tocbibind} 
\usepackage[pagebackref=true]{hyperref}
\renewcommand*\backref[1]{\ifx#1\relax \else \mbox{\textcolor{gray}{Page #1}} \fi}

\hypersetup{
	colorlinks,%
	citecolor=blue,%
	filecolor=black,%
	urlcolor=black,%
	linkcolor=red
}

\newcommand\blfootnote[1]{%
	\begingroup
	\renewcommand\thefootnote{}\footnote{#1}%
	\addtocounter{footnote}{-1}%
	\endgroup
}

\newtheorem{theorem}{Theorem}[section]
\newtheorem{thmmain}{Theorem}
\newtheorem{thmmaindual}{Theorem}
\newtheorem{lemma}[theorem]{Lemma}
\newtheorem{proposition}[theorem]{Proposition}

\theoremstyle{remark}
\newtheorem{remark}[theorem]{Remark}
\numberwithin{equation}{section}

\newcommand{\N}{\mathbb{N}}
\newcommand{\Z}{\mathbb{Z}}
\newcommand{\R}{\mathbb{R}}
\newcommand{\Rn}{\R^n}
\newcommand{\s}{\mathbb{S}}
\newcommand{\sn}{\mathbb{S}^{n-1}}
\newcommand{\sN}{\mathbb{S}^{n}}
\newcommand{\K}{\mathcal{K}}
\newcommand{\Kn}{\K^n}
\newcommand{\KN}{\K^{n+1}}

\newcommand{\hm}{\mathcal H} 
\newcommand{\MA}{\mathrm{MA}} 
\newcommand{\MAp}{\mathrm{MA}^*} 

\newcommand{\fconvs}{{\mathrm{Conv}_{\mathrm{sc}}(\R^n)}} 
\newcommand{\fconvsm}{{\mathrm{Conv}_{\mathrm{sc}}(\R^{n-1})}} 
\newcommand{\fconvsO}{{\mathrm{Conv}_{\mathrm{sc}}(\R)}} 
\newcommand{\fconvf}{{\mathrm{Conv}(\R^n; \R)}} 
\newcommand{\fconvfk}{{\mathrm{Conv}(\R^k; \R)}} 
\newcommand{\fconvfE}{{\mathrm{Conv}(E; \R)}} 
\newcommand{\fconvfF}{{\mathrm{Conv}(F; \R)}} 
\newcommand{\fconvfz}{{\mathrm{Conv}_{(o)}(\R^n; \R)}} 

\newcommand{\oZ}{\operatorname{Z}}
\newcommand{\oz}{\operatorname{z}}
\newcommand{\oV}{\operatorname{V}}

\newcommand{\ot}{\operatorname{t}}
\newcommand{\dom}{\operatorname{dom}}

\newcommand{\proj}{\operatorname{proj}} 
\newcommand{\gnom}{\operatorname{gno}} 
\newcommand{\supp}{\operatorname{supp}} 

\newcommand{\infconv}{\mathbin{\Box}} 
\newcommand{\sq}{\mathbin{\vcenter{\hbox{\rule{.3ex}{.3ex}}}}} 

\newcommand{\cA}{\operatorname{\mathcal A}}
\newcommand{\cR}{\operatorname{\mathcal R}}
\newcommand{\cT}{\operatorname{\mathcal T}}

\newcommand{\SO}{\operatorname{SO}}
\newcommand{\SOn}{\SO(n)}
\newcommand{\SOnm}{\SO(n-1)}
\newcommand{\Om}{\operatorname{O}(n-1)}
\newcommand{\On}{\operatorname{O}(n)}
\newcommand{\Ot}{\operatorname{O}(2)}
\newcommand{\Oo}{\operatorname{O}(1)}
\newcommand{\Grass}[2]{\operatorname{G}(#2,#1)}

\renewcommand{\d}{\,\mathrm{d}} 
\newcommand{\Jac}{{\operatorname{D}}} 
\newcommand{\Hess}{{\operatorname{D}}^2} 
\newcommand{\ind}{{\mathbf{I}}} 

\begin{document}
	\title{The Vectorial Hadwiger Theorem on Convex Functions}
	\author{Mohamed A.\ Mouamine and Fabian Mussnig}
	\date{}
	\maketitle
	
	\begin{abstract}
		A complete classification of continuous, dually epi-translation invariant, and rotation\linebreak equivariant valuations on convex functions is established. This characterizes the recently introduced functional Minkowski vectors, which naturally extend the classical Minkowski relations. For this, the existence of these operators with singular densities is shown, along with additional representations involving mixed Monge--Amp\`ere measures, Kubota-type formulas, and area measures of higher dimensional convex bodies. Dual results are formulated for valuations on super-coercive convex functions.
				
	\end{abstract}
	
	{
		\hypersetup{linkcolor=black}
		\small
		\tableofcontents
	}
	
	\goodbreak
	
	\normalsize

	\section{Introduction}
	\label{se:introduction}
	\paragraph{Minkowski Relations}
	For a \textit{convex body} $K\in\Kn$, where $\Kn$ is the set of non-empty, compact, convex subsets of $\Rn$, let $S_{n-1}(K,\cdot)$ denote its \emph{surface area measure}, which is a Borel measure on the \textit{Euclidean unit sphere} $\sn$. When $K$ has positive volume, this measure is given by
	\begin{equation}
	\label{eq:def_sam}
	S_{n-1}(K,\omega)=\hm^{n-1}\big(\{x \in \operatorname{bd}(K) : K \text{ has an outer unit normal at } x \text{ belonging to } \omega \}\big)
	\end{equation}
	for Borel sets $\omega\subseteq \sn$, where $\operatorname{bd}(K)$ is the boundary of $K$ and $\hm^{n-1}$ denotes the $(n-1)$-dimensional \textit{Hausdorff measure}. It is elementary to check that the first moment of this measure vanishes, that is
	\[
	\int_{\sn}z \d S_{n-1}(K,z)=o
	\]
	for every $K\in\Kn$, with $o$ denoting the origin in $\Rn$. This property essentially characterizes surface area measures and is the core of the \textit{Minkowski problem}. Indeed, Minkowski's existence theorem \cite[Theorem 8.2.2]{schneider_cb} states that a Borel measure $\mu$ on $\sn$ is the surface area measure of a full-dimensional convex body $K$, i.e.\ $\mu=S_{n-1}(K,\cdot)$, if and only if $\mu$ satisfies $\int_{\sn} z \d\mu(z)=o$ and is not concentrated on any great subsphere.
	
	\medskip
	
	Generalizing the surface area measure, the \textit{area measures} $S_j(K,\cdot)$ with $j\in\{0,\ldots,n-1\}$ arise from the Steiner-type formula
	\begin{equation}
	\label{eq:def_sj}
	S_{n-1}(K+r B^n,\cdot)=\sum_{j=0}^{n-1} \binom{n-1}{j} r^{n-1-j} S_j(K,\cdot)
	\end{equation}
	for $K\in\Kn$ and $r>0$, where we write $B^n$ for \textit{Euclidean unit ball} and 
	\[
	K+L=\{x+y:x\in K, y\in L\}
	\]
	for the \textit{Minkowski sum} of two convex bodies $K$ and $L$. Again, the first moments of these measures vanish,
	\begin{equation}
	\label{eq:mink_rel}
	\int_{\sn} z \d S_j(K,z)=o,
	\end{equation}
	which is also referred to as the \textit{Minkowski relations} (see, for example, \cite[(5.30) and (5.52)]{schneider_cb}). While this shows that the disappearance of the first moment of a Borel measure $\mu$ on $\sn$ is necessary for it to be the $j$th area measure of some convex body $K$, it is, in contrast to the case $j=n-1$, not sufficient when $j<n-1$. In particular, determining such necessary and sufficient conditions for $1<j<n-1$, known as the \textit{Christoffel--Minkowski problem}, is a central open problem in convex geometry and geometric analysis. See, for example, \cite{guan_ma_chris_mink_I,hyz_bams}. We also refer to \cite{boeroeczky_henk_pollehn,boeroeczky_lyz_cpam_2020,cabezas_hu,gardner_hug_xing_ye,hlyz_acta,milman_log_bm_jems} for related results.

	\paragraph{Valuations}
	A different perspective on the Minkowski relations opens up when we look at the properties of the integral on the left side of \eqref{eq:mink_rel}, which defines a vector-valued map $\oz\colon\Kn\to\Rn$. It is a straightforward consequence of \eqref{eq:def_sam} and \eqref{eq:def_sj} that $\oz$ is \textit{translation invariant}, meaning $\oz(K+x)=\oz(K)$ for every $K\in\Kn$ and $x\in\Rn$, and \textit{rotation equivariant}, that is $\oz(\vartheta K)=\vartheta\oz(K)$ for every $K\in\Kn$ and $\vartheta\in\SOn$. Furthermore, $\oz$ is a \textit{valuation}, which means that
	\begin{equation}
		\label{eq:val}
		\oz(K)+\oz(L)=\oz(K\cup L)+\oz(K\cap L)
	\end{equation}
	for every $K,L\in\Kn$ such that also $K\cup L\in\Kn$. In addition, $\oz$ is continuous, where we equip $\Kn$ with the \textit{Hausdorff topology}. Notably, we have the following consequence of a characterization of the \textit{intrinsic moments} by Hadwiger and Schneider \cite[Hauptsatz]{hadwiger_schneider}, where we say that a map $\oz\colon\Kn\to\Rn$ is \textit{trivial} if $\oz\equiv o$.
	
	\begin{proposition}[Hadwiger--Schneider]
		\label{prop:hadwiger_schneider}
		A map $\oz\colon \Kn\to\Rn$ is a continuous, translation invariant, rotation equivariant valuation, if and only if $\oz$ is trivial.
	\end{proposition}
	
	From today's perspective, the work of Hadwiger and Schneider characterizes the rank 1 \textit{Mink\-ow\-ski tensors}, a natural family of isometry covariant valuations that has received considerable attention (see, for example, \cite{alesker_geom_ded_1999,bernig_hug_2018,boeroeczky_domokos_solanes_jfa_2021,hug_schneider_gafa_2014}). The model for such results is undoubtedly the characterization of the \textit{intrinsic volumes} $V_j\colon\Kn\to\R$, $j\in\{0,\ldots,n\}$, which are the Minkowski tensors of rank 0. For $j<n$, the $j$th intrinsic volumes $V_j(K)$ is proportional to the total mass $S_j(K,\sn)$ and $V_n(K)$ is the usual $n$-dimensional \textit{volume} of $K\in\Kn$. While intrinsic volumes are also continuous, translation invariant valuations, they are \textit{rotation invariant}, meaning $V_j(\vartheta K)=V_j(K)$ for every $K\in\Kn$ and $\vartheta\in\SOn$. We state Hadwiger's celebrated characterization theorem \cite{hadwiger} next.
	
	\begin{theorem}[Hadwiger]
		\label{thm:hadwiger}
		A map $\oZ\colon \Kn\to\R$ is a continuous, translation and rotation invariant valuation, if and only if there exist constants $c_0,\ldots,c_n\in\R$ such that
		\begin{equation}
			\label{eq:hadwiger}
			\oZ(K)=c_0 V_0(K)+\cdots+c_n V_n(K)
		\end{equation}
		for every $K\in\Kn$.
	\end{theorem}
	
	Let us point out that it is easy to see that linear combinations of intrinsic volumes are continuous, translation and rotation invariant valuations. The strength of Hadwiger's theorem lies in establishing the highly non-trivial fact that all such valuations must be as in \eqref{eq:hadwiger}. For applications of Theorem~\ref{thm:hadwiger}, including effortless proofs of formulas in integral geometry and geometric probability, see, for example, \cite{hadwiger,klain_rota}. See also \cite{alesker_annals_1999,bernig_gafa_2009,bernig_fu_2011,haberl_parapatits_2014,ludwig_reitzner_2010} for related results. We also mention \cite{goodey_weil} which revealed that solving the Christoffel--Minkowski problem is equivalent to characterizing a special family of \textit{Minkowski valuations}, the so-called \textit{mean section bodies}.	

	\paragraph{Extensions to Convex Functions} Having considered convex bodies so far, we will now move to the more general space of (finite-valued) convex functions on $\Rn$,
	\[
	\fconvf=\{v\colon\Rn\to\R : v \text{ is convex}\}.
	\]
	Convex bodies are naturally embedded into this space by associating with each $K\in\Kn$ its \textit{support function} $h_K\in\fconvf$, given by
	\[
	h_K(x)=\max\nolimits_{y\in K}\langle x,y\rangle
	\]
	for $x\in\Rn$, where $\langle \cdot,\cdot\rangle$ denotes the usual inner product on $\Rn$. The purpose of this article is to establish a non-trivial analytic analog of Proposition~\ref{prop:hadwiger_schneider} for operators on $\fconvf$. For a natural extension of area measures and related quantities from convex bodies to convex functions, we first note that when $h_K$ is of class $C^2$ on $\Rn\setminus\{o\}$, then $S_j(K,\cdot)$ is absolutely continuous with respect to $\hm^{n-1}$ with
	\begin{equation}
	\label{eq:s_j_hm}
	\d S_j(K,z)= \binom{n-1}{j}^{-1} [\Hess h_K(z)]_j \d\hm^{n-1}(z)
	\end{equation}
	on $\sn$ for $j\in\{0,\ldots,n-1\}$. Here, $[\Hess h_K(z)]_j$ denotes the $j$th elementary symmetric function of the eigenvalues of the Hessian matrix of $h_K$ at $z\in\sn$, with the convention $[\Hess h_K(z)]_0\equiv 1$. See, for example, \cite[Corollary 2.5.3 and (4.26)]{schneider_cb}. Based on \eqref{eq:s_j_hm}, functional counterparts of the intrinsic volumes on $\fconvf$ were introduced by Colesanti, Ludwig, and Mussnig in \cite{colesanti_ludwig_mussnig_5}. The so-called \textit{functional intrinsic volumes} take the form
	\begin{equation}
		\label{eq:oZZ}
		\oV_{j,\zeta}^*(v)= \int_{\Rn} \zeta(|x|) [\Hess v(x)]_j \d x, \quad j\in\{0,\ldots,n\},
	\end{equation}
	for $v\in\fconvf\cap C^2(\Rn)$, where $|x|$ denotes the \textit{Euclidean norm} of $x\in\Rn$ and $\zeta$ is a continuous function with bounded support on $(0,\infty)$, $\zeta\in C_b({(0,\infty)})$, satisfying additional assumptions at $0^+$. When $\zeta$ continuously extends to $0$, then it follows from earlier results \cite{colesanti_ludwig_mussnig_3} that $\oV_{j,\zeta}^*$ continuously extends to all of $\fconvf$, where we consider spaces of convex functions together with the topology associated with \textit{epi-convergence}, which coincides with pointwise convergence on $\fconvf$ (see, for example, \cite[Theorem 7.17]{rockafellar_wets}). For this extension of \eqref{eq:oZZ}, the measures $[\Hess v(x)]_j \d x$ are replaced by corresponding \textit{Hessian measures} (see Section~\ref{se:hessian_measures}), which play a fundamental role in PDEs \cite{caffarelli_nirenberg_spruck,trudinger_wang_hessian_ii}. Notably, when $j=n$, we retrieve the well-known \textit{Monge--Amp\`ere measure} $\MA(v;\cdot)$, which for $v\in\fconvf\cap C^2(\Rn)$ satisfies
	\begin{equation}
		\label{eq:ma_v_c2_det}
		\d\MA(v;\cdot)=\det(\Hess v(x))\d x
	\end{equation}
	on $\Rn$. Let us furthermore emphasize that $\oV_{j,\zeta}^*(h_K)$ is a multiple of $V_j(K)$ for every $K\in\Kn$, which is a consequence of \eqref{eq:s_j_hm} and \eqref{eq:oZZ}.
	
	\medskip
	
	Functional intrinsic volumes are motivated by the extension of geometric valuation theory from convex bodies to convex functions. We mention \cite{alesker_adv_geom_2019,colesanti_ludwig_mussnig_2,colesanti_ludwig_mussnig_5,knoerr_smooth,knoerr_support,knoerr_unitarily,knoerr_singular,knoerr_ma,mussnig_adv} as examples of this rapidly developing field, and refer to \cite{brauner_hofstaetter_ortega-moreno_zonal,hug_mussnig_ulivelli_support,hug_mussnig_ulivelli_cjm,knoerr_zonal,knoerr_ulivelli_math_ann_2024} for connections and applications with the geometry of convex bodies, as well as \cite{colesanti_pagnini_tradacete_villanueva_jfa_2021,falah_rotem,ludwig_ajm_2012,milman_rotem,mussnig_ulivelli,rotem_riesz} for related results. Here, a map $\oZ\colon\fconvf\to\R$ is a \textit{valuation} if
	\begin{equation}
		\label{eq:val_fconvf}
		\oZ(v\vee w) + \oZ(v\wedge w) = \oZ(v)+\oZ(w)
	\end{equation}
	for every $v,w\in\fconvf$ such that also their pointwise maximum $v\vee w$ and minimum $v\wedge w$ are elements of $\fconvf$. Since
	\[
	h_{K\cup L}=h_K \vee h_L \quad \text{and} \quad h_{K\cap L} = h_K\wedge h_L
	\]
	for $K,L,K\cup L\in\Kn$, it follows that \eqref{eq:val_fconvf} generalizes the classical valuation property \eqref{eq:val}. In addition, $\oZ\colon\fconvf\to\R$ is called \textit{dually epi-translation invariant} if $\oZ(v+a)=\oZ(v)$ for every $v\in\fconvf$ and every affine function $a$ on $\Rn$, and \textit{rotation invariant} if $\oZ(v\circ\vartheta^{-1})=\oZ(v)$ for every $v\in\fconvf$ and $\vartheta\in\SOn$. In an effort to classify all such valuations, it was discovered that continuous extensions of \eqref{eq:oZZ} also exist for densities $\zeta$ with possible singularities at $0^+$ \cite[Theorem 1.4]{colesanti_ludwig_mussnig_5}. We therefore speak of \textit{singular Hessian valuations} and denote the associated classes of possible densities by $D_j^n\subset C_b((0,\infty))$ (see \eqref{eq:def_d_j_n} for a definition). Alternative proofs of this existence result were given in \cite{colesanti_ludwig_mussnig_6} using integral geometry and in \cite{colesanti_ludwig_mussnig_7} using mixed Monge--Amp\`ere measures. Furthermore, Knoerr showed in \cite[Theorem 1.5]{knoerr_singular} that $\oV_{j,\zeta}^*$ can be expressed as a principal value integral with respect to a Hessian measure. Additional explanations for the appearance of singularities were recently given in \cite[Section 5]{hug_mussnig_ulivelli_cjm}.
	
	\medskip
	
	Analogous to Hadwiger's characterization theorem (Theorem~\ref{thm:hadwiger}), it was shown in \cite[Theorem 1.5]{colesanti_ludwig_mussnig_5} that combinations of functional intrinsic volumes are the only continuous, dually epi-translation and rotation invariant valuations on $\fconvf$. See also \cite{colesanti_ludwig_mussnig_4}.
	
	\begin{theorem}[Hadwiger theorem on $\fconvf$]
		\label{thm:hadwiger_fconvf}
		For $n\geq 2$, a map $\oZ\colon\fconvf\to\R$ is a continuous, dually epi-translation and rotation invariant valuation, if and only if there exist functions $\zeta_0\in D_0^n,\ldots,\zeta_n\in D_n^n$ such that
		\[
		\oZ(v)=\oV_{0,\zeta_0}^*(v)+\cdots+\oV_{n,\zeta_n}^*(v)
		\]
		for every $v\in\fconvf$. For $n=1$, the same representation holds if we replace
		rotation invariance with reflection invariance.
	\end{theorem}
	\noindent
	Here, $\oZ\colon\fconvf\to\R$ is \textit{reflection invariant} if $\oZ(v)=\oZ(v^-)$ for every $v\in\fconvf$, where $v^-(x)=v(-x)$ for $x\in\Rn$. We remark that this is a necessary assumption in the one-dimensional case as otherwise, additional valuations may appear in the representation (see \cite[Corollary 6]{colesanti_ludwig_mussnig_4}). Let us furthermore point out that in Theorem~\ref{thm:hadwiger_fconvf} the valuation $\oZ$ uniquely determines the densities $\zeta_1,\ldots,\zeta_n$ but $\oV_{0,\zeta_0}^*$ is constant on $\fconvf$ and does not uniquely determine $\zeta_0$.

	\paragraph{Functional Minkowski Vectors}
	Recently, the authors introduced \textit{functional Minkowski vectors} \cite{mouamine_mussnig_1} which extend the left side of the Minkowski relations \eqref{eq:mink_rel} and take the form
	\begin{equation}
		\label{eq:ot_j_zeta_dual}
		\ot_{j,\zeta}^*(v)=\int_{\Rn} \zeta(|x|)x \,[\Hess v(x)]_j \d x,\quad j\in\{0,\ldots,n\},
	\end{equation}
	for $v\in \fconvf\cap C^2(\Rn)$, where $\zeta$ is a continuous function with compact support on $[0,\infty)$, $\zeta\in C_c({[0,\infty)})$. Using appropriate Hessian measures, we remark that \eqref{eq:ot_j_zeta_dual} continuously extends to $\fconvf$.
	Given the Minkowski relations and Proposition~\ref{prop:hadwiger_schneider}, the natural question arises as to whether these integrals vanish. Indeed, $\ot_{0,\zeta}^*\equiv o$, but the remaining operators in this family are non-trivial. It is straightforward to check that $\ot_{n,\zeta}^*(v)$ does not vanish on translates of support functions (see \cite[Section 4]{mouamine_mussnig_1}) and explicit examples for the remaining cases will be given in Section~\ref{section:retrieving_the_densities}.
	
	In Theorem~\ref{thm:main_existence}, we show the existence of functional Minkowski vectors with singular densities. For $j\in\{1,\ldots,n\}$ let
	\[
	T_j^n=\left\{\zeta\in C_b((0,\infty)) : \lim_{s\to 0^+} s^{n-j+1} \zeta(s)=0 \right\}.
	\]
	In the following we say that a vector-valued map $\oz\colon\fconvf\to\Rn$ is \textit{rotation equivariant} if $\oz(v\circ \vartheta^{-1})=\vartheta \oz(v)$ for every $v\in\fconvf$ and $\vartheta\in\SOn$. Similarly, we define \textit{$\On$ equivariance}.
	
	\begin{thmmain}
		\label{thm:main_existence}
		For $j\in\{1,\ldots,n\}$ and $\zeta\in T_j^n$, there exists a unique, continuous, dually epi-translation invariant, rotation equivariant valuation $\ot_{j,\zeta}^*\colon\fconvf\to\Rn$ such that
		\[
		\ot_{j,\zeta}^*(v)=\int_{\Rn} \zeta(|x|)x\, [\Hess v(x)]_j \d x
		\]
		for every $v\in\fconvf\cap C^2(\Rn)$. In addition, $\ot_{j,\zeta}^*$ is $\On$ equivariant.
	\end{thmmain}
	\noindent
	For the case $j=n$, we remark that $\zeta\in T_n^n$ if and only if $x\mapsto \zeta(|x|)x$ continuously extends to a continuous map with compact support on $\Rn$ (see Lemma~\ref{le:continuous_extension} below). In particular, this case was already treated in \cite{mouamine_mussnig_1}.
	
	In the course of proving Theorem~\ref{thm:main_existence}, we establish the following explicit representation of the operators $\ot_{j,\zeta}^*$, where we use the Monge--Amp\`ere-type measures $\MA_j(v;\cdot)$, which are defined through the relation
	\begin{equation}
		\label{eq:def_maj_steiner}
		\MA(v+r h_{B^n};\cdot)=\sum_{j=0}^n \binom{n}{j} r^{n-j}\MA_j(v;\cdot)
	\end{equation}
	for $v\in\fconvf$ and $r\geq 0$.
	
	\begin{thmmain}
		\label{thm:main_rep_maj}
		If $j\in\{1,\ldots,n\}$ and $\zeta\in T_j^n$, then
		\[
		\ot_{j,\zeta}^*(v)=\int_{\Rn} \alpha(|x|)x \d\MA_j(v;x)
		\] 
		for every $v\in \fconvf$, where $\alpha\in T_n^n$ is given by
		\begin{equation}
			\label{eq:alpha_zeta}
			\alpha(s)= \binom{n}{j}\left(s^{n-j} \zeta(s)+(n-j)\int_s^\infty t^{n-j-1}\zeta(t) \d t\right)
		\end{equation}
		for $s>0$.
	\end{thmmain}
	\noindent
	Note, that \eqref{eq:def_maj_steiner} and Theorem~\ref{thm:main_rep_maj} show that $\ot_{j,\zeta}^*(v)$ with $v\in\fconvf$ and $\zeta\in T_j^n$ appears as a coefficient in a Steiner formula that describes the polynomial expansion of $r\mapsto \ot_{n,\alpha}^*(v+r h_{B^n})$ with a suitable $\alpha\in T_n^n$. Conversely, since \eqref{eq:alpha_zeta} is a bijection between $T_j^n$ and $T_n^n$ (see Lemma~\ref{le:r_trans_bij} below), all coefficients in such a Steiner formula are functional Minkowski vectors. More precisely, if $\alpha\in T_n^n$, then
	\begin{equation}
		\label{eq:ot_n_alpha_steiner}
		\ot_{n,\alpha}^*(v+r h_{B^n})=\sum_{j=1}^n r^{n-j} \ot_{j,\zeta_j}^*(v)
	\end{equation}
	for every $v\in\fconvf$ and $r\geq 0$, where $\zeta_j\in T_j^n$ is given by
	\[
	\zeta_j(s)=\left(\frac{\alpha(s)}{s^{n-j}}-(n-j)\int_s^\infty \frac{\alpha(t)}{t^{n-j+1}}\d t \right)
	\]
	for $s>0$. In particular, every $\ot_{j,\zeta_j}^*(v)$ with $\zeta_j\in T_j^n$ appears exactly once as a coefficient of $r^{n-j}$ on the right side of \eqref{eq:ot_n_alpha_steiner} as $\alpha$ ranges in $T_n^n$.
	
	Let us further mention that using recent integral-geometric results for the measures $\MA_j(v;\cdot)$ from \cite{hug_mussnig_ulivelli_support}, we will see in Theorem~\ref{thm:ot_j_zeta_kubota} that Theorem~\ref{thm:main_rep_maj} is equivalent to a Kubota-type formula, thereby providing an additional representation of functional Minkowski vectors that avoids singular densities.
	
	\medskip
	
	In our main result, we give a complete characterization of the vector-valued valuations $\ot_{j,\zeta}^*$.
	
	\begin{thmmain}
		\label{thm:main_class}
		For $n\geq 3$, a map $\oz\colon \fconvf\to\Rn$ is a continuous, dually epi-translation invariant, rotation equivariant valuation, if and only if there exist unique functions $\zeta_1\in T_1^n, \ldots, \zeta_n\in T_n^n$ such that
		\begin{equation*}
			\oz(v)=\ot_{1,\zeta_1}^*(v)+\cdots+\ot_{n,\zeta_n}^*(v)
		\end{equation*}
		for every $v\in\fconvf$. For $n\leq 2$, the same representation holds if we replace rotation equivariance with $\On$ equivariance.
	\end{thmmain}
	\noindent
	Let us mention that for $n\leq 2$, $\On$ equivariance is a necessary assumption, as demonstrated in \cite[Remark 3.7]{mouamine_mussnig_1}. We furthermore emphasize that Theorem~\ref{thm:main_class} characterizes a new family of operators that emerge exclusively in the broader setting of convex functions. In contrast, the corresponding question for convex bodies, as addressed in Proposition~\ref{prop:hadwiger_schneider}, admits only a trivial solution.
	
	To prove Theorem~\ref{thm:main_class}, we will first establish a version of this result under additional smoothness assumptions in Theorem~\ref{thm:class_smooth}. For this purpose, we will also derive another representation of the functional Minkowski vectors in Section~\ref{se:functions_to_bodies}, which uses integrals with respect to area measures of associated higher dimensional convex bodies.
	
	From Theorem~\ref{thm:main_rep_maj} together with Lemma~\ref{le:g_commutes} below, we immediately obtain the following equivalent version of Theorem~\ref{thm:main_class}. Here, $C_c(\Rn;\Rn)$ denotes the set of continuous functions $\xi\colon \Rn\to\Rn$ with compact support and we say that such a function is \textit{rotation equivariant} if $\xi(\vartheta x)=\vartheta \xi(x)$ for every $x\in\Rn$ and $\vartheta\in\SOn$. Similarly, we define $\On$ equivariance.
	
	\begin{thmmain}
		\label{thm:main_class_maj}
		For $n\geq 3$, a map $\oz\colon \fconvf\to\Rn$ is a continuous, dually epi-translation invariant, rotation equivariant valuation, if and only if there exist unique rotation equivariant functions $\xi_1,\ldots,\xi_n\in C_c(\Rn;\Rn)$ such that
		\begin{equation*}
			\oz(v)=\sum_{j=1}^n \int_{\Rn} \xi_j(x)\d\MA_j(v;x)
		\end{equation*}
		for every $v\in\fconvf$. For $n\leq 2$, the same representation holds if we replace rotation equivariance with $\On$ equivariance.
	\end{thmmain}
	\noindent
	Note that the appearance of the measures $\MA_j(v;\cdot)$ in Theorem~\ref{thm:main_class_maj} is not coincidental. As a recently established family of Kubota-type formulas from \cite{hug_mussnig_ulivelli_support} shows, these measures are generated from the Monge--Amp\`ere measure (in arbitrary dimensions) together with $\SOn$, and naturally encode the action of the rotation group (see Theorem~\ref{thm:kubota_fconvf}). In view of Theorem~\ref{thm:main_rep_maj} and Theorem~\ref{thm:main_class_maj}, one might, therefore, think that a representation of functional Minkowski vectors using Hessian measures is superfluous and that singular densities only lead to unnecessary complications. However, we will see in the classification of smooth valuations in Theorem~\ref{thm:class_smooth} that it is natural to work with Hessian measures, and in particular, our approach relies on this representation.
	
	\medskip
	
	The proofs of our main results build upon and extend techniques developed in \cite{colesanti_ludwig_mussnig_5,colesanti_ludwig_mussnig_6,colesanti_ludwig_mussnig_7,colesanti_ludwig_mussnig_8}, where functional intrinsic volumes were studied. While our strategy shares greater similarities with these works in proving existence, for the characterization result we incorporate several new key ideas beyond the existing framework. Notably, we use a new Klain--Schneider theorem, which we recently established in \cite{mouamine_mussnig_1}, and results on the supports of continuous, dually epi-translation invariant valuations due to \cite{knoerr_support}. Furthermore, compared with the scalar case, additional complexities arise in the vector-valued setting, necessitating more involved proofs. Particularly, since functional Minkowski vectors vanish on radially symmetric functions, recovering the density $\zeta$ from $\ot_{j,\zeta}^*$ is significantly more intricate than for the operators $\oV_{j,\zeta}^*$. In addition, restricting vector-valued valuations to a lower-dimensional setting introduces further challenges, which we address using the functional Hadwiger theorem in an equivalent form.
	
	\paragraph{Overview of This Article}
	Dual to the setting described above, we consider valuations on super-coercive convex functions in Section~\ref{se:dual_results}. After that, we collect preliminary results in Section~\ref{se:preliminaries}. We then prove the existence results, Theorem~\ref{thm:main_existence} and Theorem~\ref{thm:main_rep_maj}, in Section~\ref{se:existence}. Following this, we establish further representations of functional Minkowski vectors in Section~\ref{se:further_reps}, which we will then use for the proof of their characterization, Theorem~\ref{thm:main_class} and Theorem~\ref{thm:main_class_maj}, in Section~\ref{se:characterization}.
	
	\section{Dual Results}
	\label{se:dual_results}
	For a function $w\colon \Rn\to [-\infty,\infty]$ we denote by
	\[
	w^*(x)=\sup\nolimits_{y\in\Rn} \big(\langle y,x\rangle - w(y)\big),\quad  x\in\Rn,
	\]
	its \textit{convex conjugate} or \textit{Legendre--Fenchel transform}. Taking the convex conjugates of the functions in $\fconvf$ results in
	\begin{align}
	\fconvs&=\{v^* : v\in\fconvf \}\label{eq:fconvs_fconvf} \\
	&=\left\{u\colon \Rn\to (-\infty,\infty] : u \text{ is l.s.c., convex, } u\not\equiv \infty \text{, } \lim_{|x|\to\infty} \frac{u(x)}{|x|}=\infty\right\},\notag
	\end{align}
	the space of proper, lower semicontinuous, super-coercive convex functions on $\Rn$. We also equip $\fconvs$ with the topology associated with epi-convergence and refer to \cite[Lemma 5]{colesanti_ludwig_mussnig_1} and \cite[Theorem 3.1]{beer_rockafellar_wets} for an easy description of epi-convergence on this space. Equivalently to \eqref{eq:fconvs_fconvf}, $u\in\fconvs$, if and only if $u^*\in\fconvf$ and, in particular, convex conjugation is a continuous bijection between these two spaces (see \cite[Theorem 11.34]{rockafellar_wets} for the continuity). Valuations on $\fconvs$ are defined analogously to \eqref{eq:val_fconvf} and, by \cite[Proposition 3.5]{colesanti_ludwig_mussnig_3}, $\oz\colon\fconvs\to\Rn$ is a valuation if and only if $v\mapsto \oz^*(v)=\oz(v^*)$ is a valuation on $\fconvf$. Furthermore, $\oz$ is called \textit{epi-translation invariant} if and only if $\oz^*$ is dually epi-translation invariant, or, equivalently,
	\[
	\oz(u\circ \tau^{-1}+c)=\oz(u)
	\]
	for every translation $\tau$ on $\Rn$, $c\in\R$ and $u\in\fconvs$. Observe that this means that $u\mapsto \oz(u)$ is invariant under translations of the epi-graph of $u$ in $\R^{n+1}$. Lastly, rotation equivariance and $\On$ equivariance of a valuation on $\fconvs$ are defined as on $\fconvf$, which is compatible with the properties of convex conjugation.
	
	Considering the above, we note that studying valuations on $\fconvf$ is equivalent to studying valuations on $\fconvs$. In particular, it is often advantageous to switch between these two spaces, and therefore, we now state the equivalent dual results to our main results from Section~\ref{se:introduction}. We begin with the existence results for functional Minkowski vectors on $\fconvs$, which are equivalent to Theorem~\ref{thm:main_existence} and Theorem~\ref{thm:main_rep_maj}.
	
	\begin{thmmaindual}
		\label{thm:main_existence_fconvs}
		For $j\in\{1,\ldots,n\}$ and $\zeta\in T_j^n$, there exists a unique, continuous, epi-translation invariant, rotation equivariant valuation $\ot_{j,\zeta}\colon\fconvs\to\Rn$ such that
		\[
		\ot_{j,\zeta}(u)=\int_{\Rn} \zeta(|\nabla u(x)|)\nabla u(x)\, [\Hess u(x)]_{n-j} \d x
		\]
		for every $u\in\fconvs\cap C_+^2(\Rn)$. In addition, $\ot_{j,\zeta}$ is $\On$ equivariant.
	\end{thmmaindual}
	\noindent
	Here, $C_+^2(\Rn)$ denotes the set of all functions in $C^2(\Rn)$ with a positive definite Hessian matrix. We remark that $u\in\fconvs\cap C_+^2(\Rn)$, if and only if $u^*\in \fconvs\cap C_+^2(\Rn) \subseteq \fconvf$. In particular, $\fconvs\cap C_+^2(\Rn)$ is, therefore, a proper subset of the convex conjugates of the elements from $\fconvf\cap C^2(\Rn)$, and the equivalence of Theorem~\ref{thm:main_existence} and Theorem~\ref{thm:main_existence_fconvs} follows from the fact that $\fconvs\cap C_+^2(\Rn)$ is a dense subset of $\fconvs$.
	
	\begin{thmmaindual}
		\label{thm:main_rep_maj_fconvs}
		If $j\in\{1,\ldots,n\}$ and $\zeta\in T_j^n$, then
		\[
		\ot_{j,\zeta}(u)=\int_{\Rn} \alpha(|y|)y \d\MAp_j(u;x)
		\] 
		for every $u\in \fconvs$, where $\alpha\in T_n^n$ is given by
		\[
		\alpha(s)= \binom{n}{j}\left(s^{n-j} \zeta(s)+(n-j)\int_s^\infty t^{n-j-1}\zeta(t) \d t\right)
		\]
		for $s>0$.
	\end{thmmaindual}
	\noindent
	Here, we use the measures $\MA_j^*(u;\cdot) = \MA_j(u^*;\cdot)$ for $u\in\fconvs$ and $j\in\{1,\ldots,n\}$. For more details, see \eqref{eq:map_steiner}. We refer to \cite[Theorem 5.5]{colesanti_ludwig_mussnig_7} for an alternative description.
	
	\medskip
	
	Next, we state equivalent formulations of Theorem~\ref{thm:main_class} and Theorem~\ref{thm:main_class_maj} concerning the classification of vector-valued valuations on $\fconvs$.
	
	\begin{thmmaindual}
		\label{thm:main_class_fconvs}
		For $n\geq 3$, a map $\oz\colon \fconvs\to\Rn$ is a continuous, epi-translation invariant, rotation equivariant valuation, if and only if there exist unique functions $\zeta_1\in T_1^n, \ldots, \zeta_n\in T_n^n$ such that
		\begin{equation*}
			\oz(u)=\ot_{1,\zeta_1}(u)+\cdots+\ot_{n,\zeta_n}(u)
		\end{equation*}
		for every $u\in\fconvs$. For $n\leq 2$, the same representation holds if we replace rotation equivariance with $\On$ equivariance.
	\end{thmmaindual}

	\begin{thmmaindual}
		For $n\geq 3$, a map $\oz\colon \fconvs\to\Rn$ is a continuous, epi-translation invariant, rotation equivariant valuation, if and only if there exist unique rotation equivariant functions $\xi_1,\ldots,\xi_n\in C_c(\Rn;\Rn)$ such that
		\begin{equation*}
			\oz(u)=\sum_{j=1}^n \int_{\Rn} \xi_j(x)\d\MAp_j(v;x)
		\end{equation*}
		for every $u\in\fconvs$. For $n\leq 2$, the same representation holds if we replace rotation equivariance with $\On$ equivariance.
	\end{thmmaindual}
		
	\section{Preliminaries}
	\label{se:preliminaries}
	In this section, we explain relevant notations and results. In addition, we prove some smaller preliminary statements. Apart from the references listed below, we refer to the monographs \cite{rockafellar_wets,schneider_cb} as basic references. Concerning the general notation and in addition to the definitions of Section~\ref{se:introduction}, we write $e_i$ for the $i$th vector of the standard orthonormal basis of $\Rn$. Furthermore, $\kappa_m$ is the $m$-dimensional volume of the Euclidean unit ball $B^m$ in $\R^m$, with the convention $\kappa_0=1$.
	
	\subsection{Hessian Measures}
	\label{se:hessian_measures}	
	For a convex function $v\in\fconvf$ let
	\[
	\partial v(x) = \{y\in\Rn : v(z)\geq v(x)+\langle y,z-x\rangle \}
	\]
	denote the \textit{subdifferential} of $v$ at $x\in\Rn$. Note that if $v$ is differentiable at $x$ with gradient $\nabla v(x)$, then $\partial v(x)=\{\nabla v(x)\}$. The Monge--Amp\`ere measures associated with $v$ is now defined as
	\begin{equation}
	\label{eq:def_ma}
	\MA(v;B)=\hm^n\big(\{y : \exists \,x\in B \text{ such that } y\in \partial v(x)\}\big)
	\end{equation}
	for every Borel set $B\subset \Rn$ (cf.\ \cite[Definition 2.1]{figalli}). Under $C^2$ assumptions on $v$ this gives \eqref{eq:ma_v_c2_det}. More generally, we associate with each $v\in\fconvf$ the family of \textit{Hessian measures} $\Phi_j(v;\cdot)$, $j\in\{0,\ldots,n\}$, which are defined through the relation 
	\begin{equation}
		\label{eq:def_phi_j}
		\hm^{n}(\{x+r y : x\in B, y \in \partial v(x)\})=\sum_{j=0}^n r^j\, \Phi_j(v;B)
	\end{equation}
	for $r\geq 0$, where $B\subset \Rn$ is again an arbitrary Borel set. In particular, this defines a family of non-negative, locally finite Borel measures on $\Rn$. Note, that $\Phi_n(v;\cdot)=\MA(v;\cdot)$. When $v$ is of class $C^2$ on some open set $O\subseteq\Rn$, then $\Phi_j(v;\cdot)$ is absolutely continuous w.r.t.\ the Lebesgue measure with
	\begin{equation}
		\label{eq:phi_j_hess}
		\d\Phi_j(v;x)=[\Hess v(x)]_j \d x
	\end{equation}
	on $O$, which generalizes \eqref{eq:ma_v_c2_det}. Cf.\ \cite[Section 8.5]{colesanti_ludwig_mussnig_3}. See also \cite{colesanti_1997,trudinger_wang_hessian_ii}. Let us point out that \eqref{eq:phi_j_hess} and the weak continuity of Hessian measures with respect to epi-convergence show that these measures are \textit{locally determined}. This means that if $v_1,v_2\in\fconvf$ are such that $v_1\equiv v_2$ on some open set $O\subseteq \Rn$, then
	\begin{equation}
	\label{eq:loc_det}
	\Phi_j(v_1;B\cap O)=\Phi_j(v_2;B\cap O)
	\end{equation}
	for every relatively compact Borel set $B\subset \Rn$. See also \cite{knoerr_ma}. 
	
	\medskip
	
	If $E$ is a $k$-dimensional linear subspace of $\Rn$, we denote by $\fconvfE$ the set of all convex functions $v\colon E\to \R$. To avoid confusion and in case necessary, we will write $\Phi_j^{(k)}(v;\cdot)$ for $v\in\fconvfE$, where the superscript indicates the dimension of the ambient space on which we consider a Hessian measure. The following result on product measures was shown in \cite[Lemma 4.8]{colesanti_ludwig_mussnig_1}
	\begin{lemma}
		\label{le:product_measure}
		If $E$ and $F$ are orthogonal, complementary subspaces of $\Rn$ with $k=\dim E$, then
		\[
		\Phi_l^{(n)}(v_E+v_F;\cdot) = \sum_{i=0 \vee (l-k-n)}^{k\wedge l} \Phi_i^{(k)}(v_E;\cdot)\otimes \Phi_{l-i}^{(n-k)}(v_F;\cdot)
		\]
		for every $v_E\in\fconvfE$, $v_F\in\fconvfF$, and $l\in\{0,\ldots,n\}$, where $(v_E+v_F)(x_E+x_F)=v_E(x_E)+v_F(x_F)$ for $x_E\in E$, $x_F\in F$.
	\end{lemma}
	
	The following consequence of the last result was used, for example, in the proof of \cite[Lemma 2.9]{colesanti_ludwig_mussnig_8}.
	
	\begin{lemma}
		\label{le:phi_j_lower_dim}
		Let $k\in\{1,\ldots,n-1\}$, $j\in\{0,\ldots,k\}$, and $v\in\fconvf$. If there exists $w\in\fconvfk$ such that
		\[
		v(x_1,\ldots,x_n)=w(x_1,\ldots,x_k)
		\]
		for every $(x_1,\ldots,x_n) \in\Rn$, then
		\[
		\d\Phi_j(v;(x_1,\ldots,x_n))=\d\Phi_j^{(k)}(w;(x_1,\ldots,x_k))\d x_{k+1}\cdots \d x_n.
		\]
	\end{lemma}
	\begin{proof}
		The statement immediately follows by applying Lemma~\ref{le:product_measure} with $E=\R^k$, $F=(\R^k)^\perp$, $v_E=w$ ,and $v_F\equiv 0$, together with the observation that 
		$\Phi_m^{(n-k)}(v_F;\cdot)\equiv 0$ for every $m\in\{1,\ldots,n-k\}$, which is an implication of \eqref{eq:phi_j_hess}.
	\end{proof}
	
	For a super-coercive convex function $u\in\fconvs$ we consider the measures 
	\begin{equation}
	\label{eq:psi_j_phi_j}
	\Psi_j(u;\cdot)=\Phi_j(u^*;\cdot)
	\end{equation}
	with $j\in\{0,\ldots,n\}$. When $u$ is of class $C_+^2$, then
	\begin{equation}
	\label{eq:psi_j_hessian}
	\int_{\Rn} b(y) \d\Psi_j(u;y)=\int_{\Rn} b(\nabla u(x)) [\Hess u(x)]_{n-j} \d x
	\end{equation}
	for every measurable $b\colon \Rn\to\R$ such that the integral above exists (cf.\ \cite[Theorem 8.2]{colesanti_ludwig_mussnig_3}). For the special case $j=n$ we also write $\Psi_n(u;\cdot)=\MAp(u;\cdot)=\MA(u^*;\cdot)$, which we call the \textit{conjugate Monge--Amp\`ere measure} associated with $u$. Similar to \eqref{eq:psi_j_hessian}, we have
	\begin{equation}
	\label{eq:map_grad}
	\int_{\Rn} b(y)\d\MAp(u;y) = \int_{\dom(u)} b(\nabla u(x))\d x
	\end{equation}
	for every measurable $b\colon \Rn\to\R$ and every $u\in\fconvs$ without additional $C^2$ assumptions, where $\dom(u)=\{x\in \Rn : u(x)<\infty\}$ denotes the \textit{domain} of $u$. Here, we also use the fact that every convex function is differentiable almost everywhere on the interior of its domain.
	
	\subsection{Mixed Monge--Amp\`ere Measures}
	The \textit{mixed Monge--Amp\`ere measure} $\MA(v_1,\ldots,v_n;\cdot)$ associated with $v_1,\ldots,v_n\in\fconvf$ is the unique measure that is symmetric in its entries, and that satisfies
	\[
	\MA(\lambda_1 w_1 + \cdots +\lambda_m w_m;\cdot)=\sum_{i_1,\ldots,i_n=1}^m \lambda_{i_1}\cdots \lambda_{i_n} \MA(w_{i_1},\ldots,w_{i_n};\cdot)
	\]
	for every $w_1,\ldots,w_m\in\fconvf$, $\lambda_1,\ldots,\lambda_m\geq 0$, and $m\in\N$. This measure generalizes the Monge--Amp\`ere measure via
	\[
	\MA(v,\ldots,v;\cdot)=\MA(v;\cdot)
	\]
	for $v\in\fconvf$. If $v_1,\ldots,v_n\in \fconvf$ are of class $C^2$ on some open set $O\subseteq \Rn$, then $\MA(v_1,\ldots,v_n;\cdot)$ is absolutely continuous w.r.t.\ the Lesbesgue measure such that
	\begin{equation}
	\label{eq:mixed_ma_dis}
	\d\MA(v_1,\ldots,v_n;\cdot)= D(\Hess v_1(x),\ldots,\Hess v_n(x))\d x
	\end{equation}
	on $O$, where for symmetric matrices $A_1,\ldots,A_n\in\R^{n\times n}$ we write $D(A_1,\ldots,A_n)$ for their \textit{mixed discriminant}. See \cite[Theorem 4.3]{colesanti_ludwig_mussnig_7}.
	
	The measures $\MA_j(v;\cdot)$ with $j\in\{0,\ldots,n\}$ and $v\in\fconvf$, which were defined in \eqref{eq:def_maj_steiner}, are special cases of mixed Monge--Amp\`ere measures since
	\begin{equation}
	\label{eq:ma_j_ma}
	\MA_j(v;\cdot)=\MA(v[j],h_{B^n}[n-j];\cdot).
	\end{equation}
	Here, $v[j]$ means that the entry $v$ is repeated $j$ times and, similarly, $h_{B^n}$ is repeated $(n-j)$ times. Furthermore, denoting the $n\times n$ identity matrix by $I_n$, we have
	\begin{equation}
		\label{eq:elem_symm_mixed_dis}
		[A]_j = \binom{n}{j} D(A[j],I_n[n-j])
	\end{equation}
	for every symmetric $A\in \R^{n\times n}$ and $j\in\{0,\ldots,n\}$. Thus, we also retrieve
	\begin{equation}
	\label{eq:phi_j_ma}
	\Phi_j(v;\cdot) = \binom{n}{j}\MA(v[j],\tfrac 12 h_{B^n}^2[n-j];\cdot)
	\end{equation}
	for every $v\in\fconvf$, where we have used that $\Hess \frac 12 h_{B^n}^2(x) = \Hess \frac 12 |x|^2 = I_n$ for $x\in\Rn$. See \cite[Section 4.2]{colesanti_ludwig_mussnig_7}.
	
	\medskip
	
	The analog of the mixed Monge--Amp\`ere measure on $\fconvs$ is the \textit{conjugate mixed Monge--Amp\`ere measure} $\MAp(u_1,\ldots,u_n;\cdot)$ which we associate to $u_1,\ldots,u_n\in\fconvs$. This is the unique measure that is symmetric in its entries such that
	\[
	\MAp((\lambda_1 \sq w_1) \infconv \cdots \infconv (\lambda_m \sq w_m);\cdot)=\sum_{i_1,\ldots,i_n=1}^m \lambda_{i_1}\cdots\lambda_{i_n} \MAp(w_{i_1},\ldots,w_{i_n};\cdot)
	\]
	for every $w_1,\ldots,w_m\in\fconvs$, $\lambda_1,\ldots,\lambda_m \geq 0$, and $m\in\N$. Here, $u_1\infconv  u_2\in\fconvs$ is the \textit{infimal convolution} or \textit{epi-sum} of $u_1,u_2\in\fconvs$, that is,
	\[
	(u_1 \infconv u_2)(x)= \inf\nolimits_{y\in\Rn} \big(u_1(x-y)+u_2(y) \big)
	\]
	for $x\in\Rn$, which means that $u_1\infconv u_2 = (u_1^*+u_2^*)^*$. Furthermore, $\lambda \sq u \in\fconvs$ denotes the \textit{epi-multiplication} of $u\in\fconvs$ with $\lambda\geq 0$, which can be defined as $\lambda \sq u = (\lambda u^*)^*$. Equivalently, for positive $\lambda$ we have
	\[
	\big(\lambda \sq u\big)(x)=\lambda\, u\left( \tfrac x\lambda \right)
	\]
	for $x\in\Rn$. In particular, it is now straightforward to see that
	\[
	\MAp(u_1,\ldots,u_n;\cdot)=\MA(u_1^*,\ldots,u_n^*;\cdot)
	\]
	for $u_1,\ldots,u_n\in\fconvs$.
	
	The convex conjugate of the support function $h_K$ with $K\in\Kn$ is the (convex) \textit{indicator function} $\ind_K\in\fconvs$ of $K$,
	\begin{equation}
		\label{eq:conj_support_indicator}
		\ind_K(x)=(h_K)^*(x)=\begin{cases}
			0\quad &\text{if } x\in K\\
			\infty\quad&\text{else}. 
	\end{cases}
	\end{equation}
	Thus, by \eqref{eq:def_maj_steiner}, the measures
	\[
	\MAp_j(u;\cdot)=\MA_j(u^*;\cdot)=\MA(u^*[j],h_{B^n}[n-j];\cdot)=\MAp(u[j],\ind_{B^n}[n-j];\cdot)
	\]
	with $u\in\fconvs$ and $j\in\{0,\ldots,n\}$ are equivalently obtained from the Steiner formula
	\begin{equation}
	\label{eq:map_steiner}
	\MAp(u\infconv (r\sq \ind_{B^n});\cdot)=\sum_{j=0}^n \binom{n}{j} r^{n-j}\MAp_j(u;\cdot)
	\end{equation}
	for $r\geq 0$.	
	
	\subsection{Valuations on Convex Functions}
	Valuations on the spaces $\fconvf$ and $\fconvs$ were already introduced in Section~\ref{se:introduction} and Section~\ref{se:dual_results}. A valuation $\oz\colon\fconvs\to\Rn$ is called \textit{epi-homogeneous} of degree $s\in\R$ if
	\[
	\oz(\lambda\sq u)=\lambda^s \oz(u)
	\]
	for every $u\in\fconvs$ and $\lambda>0$. Equivalently, the dual valuation $\oz^*$ on $\fconvf$, that is $\oz^*(v)=\oz(v^*)$ for $v\in\fconvf$, is \textit{homogeneous} of the same degree, which means that
	\[
	\oz^*(\lambda v)=\lambda^s \oz^*(v)
	\]
	for every $v\in\fconvf$ and $\lambda>0$. Examples of homogeneous valuations are given in the next result due to \cite[Proposition 4.4]{colesanti_ludwig_mussnig_3}.
	
	\begin{proposition}
		\label{prop:int_ma_is_a_val}
		Let $j\in\{0,\ldots,n\}$. If $\beta\in C_c(\Rn)$ and $v_1,\ldots,v_{n-j}\in\fconvf$, then
		\[
		v\mapsto \int_{\Rn} \beta(x)\d\MA(v[j],v_1,\ldots,v_{n-j};x)
		\]
		defines a continuous, dually epi-translation invariant valuation on $\fconvf$ that is homogeneous of degree $j$.
	\end{proposition}
	
	We will use Proposition~\ref{prop:int_ma_is_a_val} together with the following consequence of \cite[Lemma 3.4]{mouamine_mussnig_1}.
	\begin{lemma}
		\label{le:continuous_extension}
		If $\alpha\in T_n^n$, then $\xi(x)= \alpha(|x|)x$ continuously extends from $\Rn\setminus\{o\}$ to $\Rn$ with $\xi(o)=o$. In particular, $\xi\in C_c(\Rn;\Rn)$.
	\end{lemma}

	We can now establish the needed properties of the valuations that are relevant to this article.
	
	\begin{proposition}
		\label{prop:int_maj_is_a_vector-val}
		Let $j\in\{0,\ldots,n\}$. If $\alpha\in T_n^n$, then
		\begin{equation}
		\label{eq:int_maj_val}
		v\mapsto \int_{\Rn} \alpha(|x|)x\,\d\MA_j(v;x)\quad \text{and}\quad v\mapsto \int_{\Rn} \alpha(|x|)x\,\d\Phi_j(v;x)
		\end{equation}
		define continuous, dually epi-translation invariant, $\On$ equivariant, vector-valued valuations on $\fconvf$ that are homogeneous of degree $j$.
	\end{proposition}
	\begin{proof}
		It follows from \eqref{eq:ma_j_ma} and \eqref{eq:phi_j_ma} together with Proposition~\ref{prop:int_ma_is_a_val}, applied coordinate-wise, as well as Lemma~\ref{le:continuous_extension} that the maps in \eqref{eq:int_maj_val} define continuous, dually epi-translation invariant valuations that are homogeneous of degree $j$. Next, for $v\in\fconvf$ and $\vartheta\in \On$ it trivially follows from \eqref{eq:def_maj_steiner} and \eqref{eq:def_ma} that
		\[
		\MA_j(v\circ \vartheta^{-1};B)=\MA_j(v;\vartheta^{-1}B)
		\]
		for every Borel set $B\subset\Rn$. Thus,
		\[
		\int_{\Rn} \alpha(|x|)x \d\MA_j(v\circ \vartheta^{-1};x)= \int_{\Rn} \alpha(|\vartheta x|)\vartheta x \d\MA_j(v;x)=\vartheta \int_{\Rn} \alpha(|x|)x\,\d\MA_j(v;x),
		\]
		which shows the claimed $\On$ equivariance of the first map in \eqref{eq:int_maj_val}. The $\On$ equivariance of the integral with respect to $\Phi_j(v;\cdot)$ is shown analogously.
	\end{proof}

	We also need an equivalent version of the last result for valuations on super-coercive convex functions.
	
	\begin{proposition}
		\label{prop:int_maj_is_a_vector-val_fconvs}
		Let $j\in\{0,\ldots,n\}$. If $\alpha\in T_n^n$, then
		\begin{equation*}
			u\mapsto \int_{\Rn} \alpha(|y|)y\,\d\MAp_j(u;x)\quad \text{and}\quad u\mapsto \int_{\Rn} \alpha(|y|)y\,\d\Psi_j(u;x)
		\end{equation*}
		define continuous, epi-translation invariant, $\On$ equivariant, vector-valued valuations on\linebreak$\fconvs$ that are epi-homogeneous of degree $j$.
	\end{proposition}
		
	The following decomposition result for real-valued valuations was shown in \cite[Theorem 4]{colesanti_ludwig_mussnig_4} and later also in \cite{knoerr_support}, where it was stated in a more general form. Let us emphasize that the version below can be trivially applied coordinate-wise to vector-valued valuations.
	
	\begin{theorem}
		\label{thm:mcmullen_fconvf}
		If $\oZ\colon\fconvf\to\R$ is a continuous, dually epi-translation invariant valuation, then there exist continuous, dually epi-translation invariant valuations $\oZ_0,\ldots,\oZ_n\colon\fconvf\to\R$ such that $\oZ_j$ is homogeneous of degree $j$, $j\in\{0,\ldots,n\}$, and $\oZ=\oZ_0+\cdots+\oZ_n$.
	\end{theorem}

	The equivalent version on $\fconvs$ reads as follows.
	
	\begin{theorem}
		\label{thm:mcmullen_fconvs}
		If $\oZ\colon\fconvs\to\R$ is a continuous, epi-translation invariant valuation, then there exist continuous, epi-translation invariant valuations $\oZ_0,\ldots,\oZ_n\colon\fconvs\to\R$ such that $\oZ_j$ is epi-homogeneous of degree $j$, $j\in\{0,\ldots,n\}$, and $\oZ=\oZ_0+\cdots+\oZ_n$.
	\end{theorem}

	Classifications of vector-valued valuations of extremal degrees were established in \cite{mouamine_mussnig_1}. We start with the following consequence of \cite[Theorem 5.4]{mouamine_mussnig_1}, where the equivalent setting of valuations on $\fconvf$ was considered and where we say that $\oz\colon\fconvsO\to\R^1$ is \textit{reflection equivariant} if $\oz$ is $\Oo$ equivariant.
	
	\begin{theorem}
		\label{thm:class_0_hom}
		For $n\geq 2$, a map $\oz\colon\fconvs\to\Rn$ is a continuous, epi-translation invariant, rotation equivariant valuation that is epi-homogeneous of degree $0$, if and only if it is identically $o$. For $n= 1$, the same representation holds if $\oz$ is reflection equivariant.
	\end{theorem}

	Next, we state \cite[Theorem 5.5]{mouamine_mussnig_1}, which concerns valuations of top degree. Here, we also want to point out that if $\zeta\in T_n^n$, then
	\[
	\ot_{n,\zeta}(u)=\int_{\Rn} \zeta(|y|)y\d\Psi_n(u;y)
	\]
	for every $u\in\fconvs$.
	
	\begin{theorem}
		\label{thm:class_n_hom}
		For $n\geq 3$, a map $\oz\colon\fconvs\to\Rn$ is a continuous, epi-translation invariant, rotation equivariant valuation that is epi-homogeneous of degree $n$, if and only if there exists $\zeta\in T_n^n$ such that
		\[
		\oz(u)=\ot_{n,\zeta}(u)
		\]
		for every $u\in\fconvs$. For $n\leq 2$, the same representation holds if we replace rotation equivariance with $\On$ equivariance.
	\end{theorem}

	The following result is a consequence of a Klain--Schneider theorem for vector-valued valuations on $\fconvs$, which is due to \cite[Theorem 3.6 and Remark 3.7]{mouamine_mussnig_1}. Here, we say that $\oz\colon\fconvs\to\Rn$ is \textit{simple} if $\oz(u)=o$ for every $u\in\fconvs$ such that $\dim(\dom(u))<n$.
	
	\begin{theorem}
		\label{thm:class_simple}
		For $n\geq 3$, a map $\oz\colon\fconvs\to\Rn$ is a continuous, epi-translation invariant, rotation equivariant, simple valuation, if and only if there exists $\zeta\in T_n^n$ such that
		\[
		\oz(u)=\ot_{n,\zeta}(u)
		\]
		for every $u\in\fconvs$. For $n\leq 2$, the same representation holds if we replace rotation equivariance with $\On$ equivariance.
	\end{theorem}
	
	Let us close this section with the following simple result due to \cite[Lemma 3.5]{mouamine_mussnig_1}.
	\begin{lemma}
		\label{le:g_commutes}
		Let $\xi\in C_c(\Rn;\Rn)$. For $n\geq 3$, the map $\xi$ satisfies
		\begin{equation}
			\label{eq:g_commutes}
			\xi(\vartheta x)=\vartheta \xi(x)
		\end{equation}
		for every $\vartheta\in\SOn$ and $x\in\Rn\setminus\{o\}$, if and only if there exists $\alpha\in T_n^n$ such that
		\[
		\xi(x)=\alpha(|x|)x
		\]
		for every $x\in\Rn\setminus\{o\}$. For $n\leq 2$, the same representation holds if $\xi$ is assumed to commute with $\On$ instead.
	\end{lemma}	
	
	\section{Existence}
	\label{se:existence}
	In this section, we prove Theorem~\ref{thm:main_existence}, that is, the existence of the operators $\ot_{j,\zeta}^*$, $\zeta\in T_j^n$, in terms of singular Hessian integrals. This goes hand in hand with establishing Theorem~\ref{thm:main_rep_maj}.
	
	\subsection{Integral Transforms Between the Classes \texorpdfstring{$\boldsymbol{T_j^n}$}{Tjn}}
	For $\zeta\in C_b((0,\infty))$ and $s>0$, let
	\[
	\cR \zeta(s)=s\zeta(s)+\int_s^\infty \zeta(t)\d t.
	\]
	It is easy to see that also $\cR \zeta\in C_b((0,\infty))$ and this integral transform was introduced in \cite{colesanti_ludwig_mussnig_6} in the context of Cauchy--Kubota formulas for functional intrinsic volumes. For $l\in\N$ we denote
	\[
	\cR^l \zeta = \underbrace{(\cR\circ \cdots\circ\cR)}_{l} \zeta
	\]
	and set $\cR^0 \zeta=\zeta$. It was shown in \cite[Lemma 3.6]{colesanti_ludwig_mussnig_6} that
	\begin{equation}
	\label{eq:r_l}
	\cR^l\zeta (s)=s^l \zeta(s)+l\int_s^\infty t^{l-1}\zeta(t)\d t
	\end{equation}
	for $s>0$. Using integration by parts (cf.\ \cite[Lemma 3.8]{colesanti_ludwig_mussnig_6} and its proof), it is furthermore straightforward to check that the inverse operation to $\cR^l$ is given by
	\begin{equation}
	\label{eq:r_inverse}
	\cR^{-l}\rho(s)=(\cR^{-1})^l\rho(s)=\frac{\rho(s)}{s^l}-l\int_s^\infty \frac{\rho(t)}{t^{l+1}}\d t
	\end{equation}
	for $s>0$, where $\rho\in C_b((0,\infty))$. Notably, for $j\in\{0,\ldots,n\}$ and $l\in\{0,\ldots,n-j\}$, the transform $\cR^l$ is a bijection between $D_j^n$ and $D_j^{n-l}$ \cite[Lemma 3.8]{colesanti_ludwig_mussnig_6}, where
	\begin{equation}
		\label{eq:def_d_j_n}
		D_j^n=\left\{\zeta\in C_b((0,\infty)): \lim_{s\to 0^+} s^{n-j}\zeta(s)=0, \lim_{s\to 0} \int_s^\infty t^{n-j-1}\zeta(t)\d t \text{ exists and is finite}\right\}
	\end{equation}
	for $j\in\{0,\ldots,n-1\}$.	In addition, $D_n^n$ is the set of $\zeta\in C_b((0,\infty))$ such that $\zeta(0)=\lim_{s\to 0^+} \zeta(s)$ exists and is finite. In the following, we want to show that $\cR^l$ is also a bijection between $T_j^n$ and $T_j^{n-l}$. To this end, we first need the following result, the proof of which works in the same way as the proof of \cite[Lemma 3.7]{colesanti_ludwig_mussnig_6}. Recall that for $j\in\{1,\ldots,n\}$ the set $T_j^n$ consists of all $\zeta\in C_b((0,\infty))$ such that $\lim_{s\to 0^+} s^{n-j+1}\zeta(s)=0$.
	
	\begin{lemma}
	\label{le:lhospital}
	Let $j\in\{1,\ldots,n-1\}$. If $\zeta\in T_j^n$, then
	\begin{equation}
	\label{eq:lim_s_n-j_int_zeta}
	\lim_{s\to 0^+} s^{n-j}\int_s^\infty \zeta(t)\d t = 0. 
	\end{equation}
	Furthermore, if $\rho\in T_j^{n-1}$, then
	\[
	\lim_{s\to 0^+} s^{n-j+1}\int_s^\infty \frac{\rho(t)}{t^2}\d t = 0.
	\]
	\end{lemma}
	\begin{proof}
	Let $\zeta\in T_j^n$. If $\lim_{s\to 0^+} \int_s^\infty \zeta(s)\d s$ exists and is finite, then \eqref{eq:lim_s_n-j_int_zeta} is immediate. In the other case, it follows from L'Hospital's rule and the definition of $T_j^n$ that
	\[
	\lim_{s\to 0^+}\left|s^{n-j}\int_s^\infty \zeta(t)\d t\right|\leq \lim_{s\to 0^+} \frac{\int_s^\infty |\zeta(t)|\d t}{\frac{1}{s^{n-j}}} = \lim_{s\to 0^+} \frac{|\zeta(s)|}{\frac{n-j}{s^{n-j+1}}}=\lim_{s\to 0^+} \frac{|s^{n-j+1}\zeta(s)|}{n-j} = 0.  
	\]
	Next, let $\rho\in T_j^{n-1}$. Again, if $\lim_{s\to 0^+} \int_s^\infty \frac{\rho(t)}{t^2}\d t$ exists and is finite, then the statement is trivial. Otherwise, we argue similarly to above to obtain
	\[
	\lim_{s\to 0^+} \left|s^{n-j+1}\int_s^\infty \frac{\rho(t)}{t^2}\d t \right|\leq \lim_{s\to 0^+} \frac{\int_s^\infty \left|\frac{\rho(t)}{t^2}\right|\d t}{\frac{1}{s^{n-j+1}}} = \lim_{s\to 0^+} \frac{\left|\frac{\rho(s)}{s^2}\right|}{\frac{n-j+1}{s^{n-j+2}}}=\lim_{s\to 0^+} \frac{|s^{n-j}\rho(s)|}{n-j+1}=0.
	\]
	\end{proof}
	
	We can now establish the relevant properties of the integral transform $\cR$.
	
	\begin{lemma}
		\label{le:r_trans_bij}
		For $j\in\{1,\ldots,n\}$ and $l\in\{0,\ldots,n-j\}$, the map $\cR^l\colon T_j^n\to T_j^{n-l}$ is a bijection.
	\end{lemma}
	\begin{proof}
		Since we already know that $\cR^l$ is invertible, it remains to prove that $\cR^l$ takes elements from $T_j^n$ to $T_j^{n-l}$ and, similarly, that \eqref{eq:r_inverse} maps from $T_j^{n-l}$ to $T_j^n$. Furthermore, the case $j=n$ (and thus $l=0$) is trivial, and we therefore assume that $j\in\{1,\ldots,n-1\}$.
		
		We start by showing that $\zeta\in T_j^n$ implies $\cR \zeta\in T_j^{n-1}$. First, we note that $\cR \zeta$ is clearly a continuous function with bounded support. Since the definition of $T_j^n$ and Lemma~\ref{le:lhospital} show that
		\[
		\lim_{s\to 0^+} s^{n-j}\cR \zeta(s)=\lim_{s\to 0^+}\left( s^{n-j+1}\zeta(s)+s^{n-j}\int_s^\infty \zeta(t) \d t \right) = 0,
		\]
		it follows that $\cR \zeta\in T_j^{n-1}$.
		
		Next, let $\rho\in T_j^{n-1}$ be given. We need to show that $\cR^{-1}\rho \in T_j^n$. Again, it is straightforward to see that $\cR^{-1}\rho\in C_b((0,\infty))$. Furthermore, it follows from the definition of $T_j^{n-1}$ together with the second part of Lemma~\ref{le:lhospital} that
		\[
		\lim_{s\to 0^+} s^{n-j+1}\cR^{-1}\rho(s)=\lim_{s\to 0^+}\left(s^{n-j}\rho(s)-s^{n-j+1}\int_s^\infty \frac{\rho(t)}{t^2} \d t\right)=0.
		\]
		Hence, $\cR^{-1}\rho\in T_j^n$.
		
		Finally, since the definition of $T_j^n$ only depends on the difference $(n-j)$, it easily follows by induction that $\cR^l \zeta \in T_j^{n-l}$ for $\zeta\in T_j^n$ and $l\in\{0,\ldots,n-j\}$, where the case $l=0$ is trivial. Similarly, we conclude that $\cR^{-l}$ maps elements from $T_j^{n-l}$ to $T_j^n$.
	\end{proof}
	
	\begin{remark}
		While the results of this section are similar to those of \cite[Section 3.2]{colesanti_ludwig_mussnig_6}, the authors are not aware that the former follow from the latter. In this context, it should be noted that $D_j^n\subset T_j^n$ and $D_j^{n+1}\subset T_j^n$ clearly hold, and we emphasize that these inclusions are strict. 
	\end{remark}

	\subsection{Connecting Hessian Measures and Mixed Monge--Amp\`ere Measures}
	\label{se:connecting_hessian_measures}
	The purpose of this section is to show how integrals of $x\mapsto \zeta(|x|)x$ with respect to the Hessian measures $\Phi_j(v;\cdot)$ can be rewritten in terms of the mixed Monge--Amp\`ere measures $\MA_j(v;\cdot)$. We note that results of this type have been established for functional intrinsic volumes in \cite[Section 6]{colesanti_ludwig_mussnig_7} and that this section is based on the same approach.
	
	We will frequently use the support function of the Euclidean unit ball, $h_{B^n}(x)=|x|$ for $x\in\Rn$. Let us remark that this function is twice continuously differentiable on $\Rn\setminus\{o\}$ with
	\[
	\Hess h_{B^n}(x)=\frac{1}{|x|}\left(I_n - \frac{x}{|x|}\otimes\frac{x}{|x|}\right)
	\]
	for $x\neq o$, where $y\otimes z$ denotes the tensor product of $y,z\in\Rn$.
	For the next result, which is due to \cite[Lemma 6.4]{colesanti_ludwig_mussnig_7}, we therefore emphasize that integration with respect to the Lebesgue measure on $\Rn$ is considered throughout the following and that it is sufficient for integrands to be defined a.e.\ on $\Rn$.
	
	\begin{lemma}
		\label{le:int_xi_exists}
		Let $k\in\{1,\ldots,n\}$. If $\varphi\in C_b((0,\infty))$ is such that $\lim_{s\to 0^+} s^{k-1} \varphi(s)$ exists and is finite, then the integral
		\[
		\int_{\Rn} \big|\varphi(|x|)\, D(\Hess v_1(x),\ldots,\Hess v_k(x),\Hess h_{B^n}(x)[n-k])\big|\d x
		\]
		is well-defined and finite for every $v_1,\ldots,v_k\in C^2(\Rn)$.
	\end{lemma}
	
	A straightforward consequence of Lemma~\ref{le:int_xi_exists} is the following variant of this result, which shows that the relevant integrals of this section are well-defined.
	
	\begin{lemma}
		\label{le:int_mixed_dis_finite}
		Let $k\in\{1,\ldots,n-1\}$. If $\psi\in C_b((0,\infty))$ is such that $\lim_{s\to 0^+} s^k \psi(s)$ exists and is finite, then the integral
		\[
		\int_{\Rn} \big|\psi(|x|)x\, D(\Hess v_1(x),\ldots,\Hess v_k(x),\Hess h_{B^n}(x)[n-k])\big|\d x
		\]
		is well-defined and finite for every $v_1,\ldots,v_k\in C^2(\Rn)$.
	\end{lemma}
	\begin{proof}
		Let $\psi\in C_b((0,\infty))$ be as in the statement and set $\varphi(s)=s\psi(s)$ for $s>0$. Clearly, $\lim_{s\to 0^+} s^{k-1}\varphi(s)=\lim_{s\to 0^+} s^k\psi(s)$ and thus, the result follows from Lemma~\ref{le:int_xi_exists}.
	\end{proof}
	
	Next, we show that we can replace $\Hess h_{B^n}$ by $\Hess \frac 12 h_{B^n}^2=I_n$ in the mixed discriminant, if in addition we replace the density $\psi$ with $\cR^{-1}\psi$. For this, we first need two lemmas. We start with \cite[Lemma 6.2]{colesanti_ludwig_mussnig_7}.
	
	\begin{lemma}
		\label{le:int_symmetric}
		If $v_0,\ldots,v_n\in C^2(\Rn)$ are such that at least one of the functions has compact support, then
		\[
		\int_{\Rn} v_0(x)\,D(\Hess v_1(x),\ldots,\Hess v_n(x))\d x = \int_{\Rn} v_n(x)\,D(\Hess v_1(x),\ldots,\Hess v_{n-1}(x),\Hess v_0(x))\d x.
		\]
	\end{lemma}
	
	We will also use the following reformulation of \cite[Lemma 6.3]{colesanti_ludwig_mussnig_7}. Cf.\ \cite[(3.3)]{colesanti_ludwig_mussnig_7}.
	
	\begin{lemma}
		\label{le:int_mixed_dis_jac_vanish}
		Let $v_1,\ldots,v_{n-1}\in C^2(\Rn)$ and let $F\colon\Rn\to\Rn$ be a continuously differentiable vector field with compact support. If $F$ is such that its Jacobian matrix $\Jac F$ is symmetric, then
		\[
		\int_{\Rn} D(\Hess v_1(x),\ldots,\Hess v_{n-1}(x),\Jac F(x))\d x = 0.
		\]
	\end{lemma}
	
	\begin{lemma}
		\label{le:int_swap_transform}
		Let $k\in\{1,\ldots,n-1\}$ and $\varepsilon>0$. If $v_1,\ldots,v_k\in C^2(\Rn)$ are such that $\Hess v_1 \equiv 0$ on $\varepsilon B^n$, then
		\begin{align}
			\begin{split}
			\label{eq:int_zeta_inverse_r}
			\int_{\Rn} \psi(|x|)x\,&D(\Hess v_1(x),\ldots,\Hess v_k(x),\Hess h_{B^n}(x)[n-k])\d x\\
			&=\int_{\Rn} (\cR^{-1}\psi)(|x|) x\, D(\Hess v_1(x),\ldots,\Hess v_k(x),\Hess h_{B^n}(x)[n-k-1],I_n)\d x
			\end{split}
		\end{align}
		for every $\psi\in C_b^2((0,\infty))$.
	\end{lemma}
	\begin{proof}
		Since $\Hess v_1 \equiv 0$ on $\varepsilon B^n$, also the mixed discriminants in both integrals in \eqref{eq:int_zeta_inverse_r} vanish on this set. Together with the compact support of $\psi$, this shows that both integrals in \eqref{eq:int_zeta_inverse_r} are well-defined and finite. We will now prove the vector-valued equation \eqref{eq:int_zeta_inverse_r} coordinate-wise. For the $i$th coordinate, $i\in\{1,\ldots,n\}$, it follows from Lemma~\ref{le:int_symmetric} that \eqref{eq:int_zeta_inverse_r} is equivalent to
		\begin{align*}
			\int_{\Rn} |x| \, &D(\Hess v_1(x),\ldots,\Hess v_k(x),\Hess h_{B^n}(x)[n-k-1],\Hess (\psi(|x|) x_i)) \d x\\
			&=\int_{\Rn} \frac{|x|^2}{2} \, D(\Hess v_1(x),\ldots,\Hess v_k(x),\Hess h_{B^n}(x)[n-k-1],\Hess ((\cR^{-1}\psi)(|x|)x_i))\d x.
		\end{align*}
		Here we have used that the mixed discriminants vanish on $\varepsilon B^n$ and thus, we can replace $h_{B^n}$, $x\mapsto \psi(|x|)x_i$, and $x\mapsto (\cR^{-1}\psi)(|x|)x_i$ with suitable functions in $C^2(\Rn)$ without changing the values of the integrals. By the multilinearity of mixed discriminants, we therefore need to show that
		\begin{align}
		\begin{split}
			\label{eq:int_mixed_det_multilin_vanish}
			\int_{\Rn} D\Big(&\Hess v_1(x),\ldots,\Hess v_k(x), \Hess h_{B^n}(x)[n-k-1],\\
			&\quad|x|\,\Hess (\psi(|x|) x_i)-\frac{|x|^2}{2}\,\Hess ((\cR^{-1}\psi)(|x|)x_i)\Big) \d x = 0.
		\end{split}
		\end{align}
		Now observe that
		\begin{align*}
			\Hess (\psi(|x|)x_i) &= \left(\frac{\partial }{\partial x_m}\left( \frac{\psi'(|x|)}{|x|}x_i x_l + \delta_{il} \psi(|x|)\right)\right)_{l,m=1,\ldots,n}\\
			&= \left(\left(\frac{\psi''(|x|)}{|x|^2}-\frac{\psi'(|x|)}{|x|^3} \right)x_i x_l x_m + \frac{\psi'(|x|)}{|x|}(\delta_{lm} x_i +\delta_{im} x_l  + \delta_{il} x_m)\right)_{l,m=1,\ldots, n}\\
			&=\left(\frac{\psi''(|x|)}{|x|^2}-\frac{\psi'(|x|)}{|x|^3} \right)x_i (x\otimes x) + \frac{\psi'(|x|)}{|x|} \left( x_i I_n + x\otimes e_i + e_i \otimes x\right)
		\end{align*}
		for every $x\in\Rn\setminus\{o\}$, where $\delta_{pq}$ denotes the usual Kronecker delta. Since $\psi$ has bounded support, we have
		\[
		\frac{\d}{\d s}(\cR^{-1}\psi)(s)= \frac{\d}{\d s} \left(\frac{\psi(s)}{s}-\int_s^\infty \frac{\psi(t)}{t^2}\d t \right) = \frac{\psi'(s)}{s}
		\]
		for $s>0$, and therefore, similar to the above,
		\begin{align*}
			\Hess ((\cR^{-1}\psi)(|x|)x_i) &= \left(\frac{\partial }{\partial x_m}\left(\frac{\psi'(|x|)}{|x|^2}x_i x_l + \delta_{il} (\cR^{-1}\psi)(|x|)  \right)\right)_{l,m=1,\ldots,n}\\
			&=\left(\left(\frac{\psi''(|x|)}{|x|^3}-2\frac{\psi'(|x|)}{|x|^4}\right) x_i x_l x_m + \frac{\psi'(|x|)}{|x|^2}(\delta_{lm} x_i  +\delta_{im} x_l +\delta_{il} x_m)\right)_{l,m=1,\ldots,n}\\
			&= \left(\frac{\psi''(|x|)}{|x|^3}-2\frac{\psi'(|x|)}{|x|^4}\right) x_i (x\otimes x)+ \frac{\psi'(|x|)}{|x|^2} \left( x_i I_n + x\otimes e_i + e_i \otimes x\right)
		\end{align*}		
		for every $x\in\Rn\setminus\{o\}$. Hence,
		\begin{align*}
		|x| \Hess(\psi(|x|)x_i)-&\frac{|x|^2}{2} \Hess ((\cR^{-1}\psi)(|x|)x_i)\\
		&= \frac 12 \frac{\psi''(|x|)}{|x|}x_i (x\otimes x) + \frac 12 \psi'(|x|)\left( x_i I_n + x\otimes e_i + e_i \otimes x\right)\\
		&= \frac 12 \Jac F(x) 
		\end{align*}
		for $x\in\Rn\setminus\{o\}$, where $F\colon \Rn\setminus\{o\}\to\Rn$ is the vector field
		\[
		F(x)= \psi'(|x|)x_i\,x + \left(\psi(|x|)|x|+\int_{|x|}^\infty \psi(t)\right) e_i.
		\]
		Thus, we can rewrite the left side of \eqref{eq:int_mixed_det_multilin_vanish} as
		\[
		\int_{\Rn} D\Big(\Hess v_1(x),\ldots,\Hess v_k(x), \Hess h_{B^n}(x)[n-k-1],\tfrac 12 \Jac F(x)\Big) \d x.
		\]
		Using again that we may replace the integrands in a neighborhood of the origin, it follows from Lemma~\ref{le:int_mixed_dis_jac_vanish} that this integral vanishes, which completes the proof.
	\end{proof}
	
	Now that we have shown the desired equation under additional regularity assumptions, we use standard approximation arguments to remove said assumptions.
	
	\begin{proposition}
		\label{prop:int_swap_once}
		Let $k\in\{1,\ldots,n-1\}$. If $\psi\in T_{n-k+1}^n$, then
		\begin{align*}
		\int_{\Rn} \psi(|x|)x\,& D(\Hess v_1(x),\ldots,\Hess v_k(x),\Hess h_{B^n}(x)[n-k])\d x\\
		&= \int_{\Rn} (\cR^{-1}\psi)(|x|)x\, D(\Hess v_1(x),\ldots,\Hess v_k(x),\Hess h_{B^n}(x)[n-k-1],I_n)\d x
		\end{align*}
		for every $v_1,\ldots,v_k\in C^2(\Rn)$.
	\end{proposition}
	\begin{proof}
		We start with the observation that since $\psi\in T_{n-k+1}^n\subset C_b((0,\infty))$, there exists $\gamma>0$ such that $\psi(s)=\cR^{-1}\psi(s)=0$ for every $s\geq \gamma$. Assume first that there exists $\varepsilon>0$ such that $\Hess v_1\equiv 0$ on $\varepsilon B^n$. Let $\psi_\varepsilon\in C_b((0,\infty))$ be such that $\psi_\varepsilon\equiv \psi$ on $[\varepsilon,\infty)$ and $\psi_\varepsilon\equiv 0$ on $(0,\frac{\varepsilon}{2}]$. Note that this implies that also \[
		\cR^{-1}\psi_{\varepsilon}\equiv \cR^{-1}\psi
		\]
		on $[\varepsilon,\infty)$. For $\delta>0$ we can find $\psi_{\varepsilon,\delta}\in C_b^2((0,\infty))$ such that $\psi_{\varepsilon,\delta}\equiv 0$ on $(0,\frac{\varepsilon}{2}] \cup [\gamma+\delta,\infty)$ and such that $\psi_{\varepsilon,\delta}$ converges uniformly to $\psi_{\varepsilon}$ on $(\frac{\varepsilon}{2},\gamma+\delta)$ as $\delta\to 0^+$. In particular, this gives uniform convergence of $\psi_{\varepsilon,\delta}$ to $\psi_{\varepsilon}$ and furthermore of $\cR^{-1}\psi_{\varepsilon,\delta}$ to $\cR^{-1}\psi_{\varepsilon}$ on $(0,\infty)$ as $\delta\to 0^+$. Considering the assumed condition on $v_1$ and observing that each of the following integrands is continuous with compact support, it now follows from Lemma~\ref{le:int_swap_transform} that
		\begin{align*}
			\int_{\Rn} \psi(|x|)x\, & D(\Hess v_1(x),\ldots,\Hess v_k(x),\Hess h_{Bn}(x)[n-k])\d x\\
			&=\int_{\Rn} \psi_{\varepsilon}(|x|)x\, D(\Hess v_1(x),\ldots,\Hess v_k(x),\Hess h_{B^n}(x)[n-k])\d x\\
			&=\lim_{\delta\to 0^+} \int_{\Rn} \psi_{\varepsilon,\delta}(|x|)x\, D(\Hess v_1(x),\ldots,\Hess v_k(x),\Hess h_{B^n}(x)[n-k])\d x\\
			&=\lim_{\delta\to 0^+} \int_{\Rn} (\cR^{-1}\psi_{\varepsilon,\delta})(|x|)x\, D(\Hess v_1(x),\ldots,\Hess v_k(x),\Hess h_{B^n}(x)[n-k-1],I_n)\d x\\
			&=\int_{\Rn} (\cR^{-1}\psi_{\varepsilon})(|x|)x\, D(\Hess v_1(x),\ldots,\Hess v_k(x),\Hess h_{B^n}(x)[n-k-1],I_n)\d x\\
			&=\int_{\Rn} (\cR^{-1}\psi)(|x|)x\, D(\Hess v_1(x),\ldots,\Hess v_k(x),\Hess h_{B^n}(x)[n-k-1],I_n)\d x.
		\end{align*}
		
		For general $v_1\in C^2(\Rn)$, we first observe that we may assume $v_1(o)=0$ and $\nabla v_1(o)=0$ since the equality to be shown depends only on the second derivatives of $v_1$. Thus, there exist $\beta,\varepsilon_0>0$ such that
		\begin{equation}
		\label{eq:v_1_estimates}
		|v_1(x)|\leq \beta|x|^2\quad \text{and}\quad |\nabla v_1(x)|\leq \beta|x|
		\end{equation}
		for $|x|\leq 2 \varepsilon_0 B^n$. Now let $\varphi\in C^2([0,\infty))$ be such that $\varphi(t)=0$ for $t\in [0,1]$ and $\varphi(t)=1$ for $t\in[2,\infty)$, and define $v_{1,\varepsilon}\in C^2(\Rn)$ as
		\[
		v_{1,\varepsilon}(x)=v_1(x)\,\varphi\left(\frac{|x|}{\varepsilon}\right)
		\]
		for $x\in\Rn$ and $\varepsilon\in(0,\varepsilon_0)$. By the choice of $\varphi$ we clearly have $\Hess v_{1,\varepsilon}\equiv 0$ on $\varepsilon B^n$ for every $\varepsilon\in(0,\varepsilon_0)$ and $\Hess v_{1,\varepsilon}(x)=\Hess v_{1}(x)$ for $|x|\geq 2 \varepsilon$. In particular, $\Hess v_{1,\varepsilon}$ converges to $\Hess v_1$ pointwise on $\Rn\setminus\{o\}$ as $\varepsilon\to 0^+$. Furthermore, it is straightforward to compute that
		\begin{align*}
		\Hess v_{1,\varepsilon}(x)&= \varphi\left(\frac{|x|}{\varepsilon}\right) \Hess v_1(x)+\frac{1}{\varepsilon|x|}\, \varphi'\left(\frac{|x|}{\varepsilon}\right) (\nabla v_1(x)\otimes x + x\otimes \nabla v_1(x)+v_1(x)\,I_n)\\
		&\quad +\frac{1}{\varepsilon^2 |x|^2}\, \varphi''\left(\frac{|x|}{\varepsilon}\right) v_1(x) ( x\otimes x) - \frac{1}{\varepsilon |x|^3}\, \varphi'\left(\frac{|x|}{\varepsilon}\right) v_1(x) (x\otimes x)
		\end{align*}
		for $x\in\Rn\setminus\{o\}$. Since for $|x|<2\varepsilon$ we have $\frac{1}{\varepsilon}<\frac{1}{2|x|}$, this shows together with \eqref{eq:v_1_estimates} and the above observations that $\Hess v_{1,\varepsilon}$ is uniformly bounded on $\gamma B^n$ for $\varepsilon\in(0,\varepsilon_0)$. Next, since $\psi\in T_{n-k+1}^n$, we have $\lim_{s\to 0^+} s^k \psi(s)=0$. Moreover, since $T_{n-k+1}^n=T_{n-k}^{n-1}$, it follows from Lemma~\ref{le:r_trans_bij} that $\cR^{-1}\psi\in T_{n-k}^n$ and therefore $\lim_{s\to 0^+} s^{k+1}\cR^{-1}\psi(s)=0$. Thus, by Lemma~\ref{le:int_mixed_dis_finite} and the dominated convergence theorem, together with the first part of the proof, we obtain
		\begin{align*}
			\int_{\Rn} &\psi(|x|)x\, D(\Hess v_1(x),\ldots,\Hess v_k(x),\Hess h_{B^n}(x)[n-k])\d x\\
			&=\lim_{\varepsilon\to 0^+} \int_{\Rn} \psi(|x|)x \, D(\Hess v_{1,\varepsilon}(x),\Hess v_2(x),\ldots,\Hess v_k(x),\Hess h_{B^n}(x)[n-k])\d x\\
			&=\lim_{\varepsilon\to 0^+} \int_{\Rn} (\cR^{-1}\psi)(|x|)x\,D(\Hess v_{1,\varepsilon}(x),\Hess v_2(x),\ldots,\Hess v_k(x),\Hess h_{B^n}(x)[n-k-1],I_n)\d x\\
			&=\int_{\Rn} (\cR^{-1}\psi)(|x|)x\, D(\Hess v_1(x),\ldots,\Hess v_k(x),\Hess h_{B^n}(x)[n-k-1],I_n)\d x,
		\end{align*}
		which concludes the proof.
	\end{proof}
	
	We now come to the main result of this section, which we obtain from an iterated application of the previous Proposition.
	
	\begin{proposition}
		\label{prop:int_swap}
		Let $j\in\{1,\ldots,n-1\}$. If $\zeta\in T_j^n$, then
		\[
		\int_{\Rn} \zeta(|x|)x \, [\Hess v(x)]_j \d x = \binom{n}{j} \int_{\Rn} (\cR^{n-j}\zeta)(|x|) x\, D(\Hess v(x)[j],\Hess h_{B^n}(x)[n-j])\d x
		\]
		for every $v\in C^2(\Rn)$.
	\end{proposition}
	\begin{proof}
		Let $\zeta\in T_j^n$ be given. For every $k\in\{j,\ldots,n-1\}$ it follows from Lemma~\ref{le:r_trans_bij} together with the fact that $T_j^k=T_{n-k+j}^n\subseteq T_{n-k+1}^n$, that $\cR^{n-k} \zeta\in T_{n-k+1}^n$. Thus, applying Proposition~\ref{prop:int_swap_once} $(n-j)$ times with $\psi=\cR^{n-k}\zeta$, $v_1\ldots,v_j=v$, $v_{j+1},\ldots,v_k=\frac 12 h_{B^n}^2$, and $k\in\{j,\ldots,n-1\}$, shows that
		\begin{align*}
			\int_{\Rn} (\cR^{n-j}\zeta)(|x|)x\,& D(\Hess v(x)[j],\Hess h_{B^n}(x)[n-j])\d x\\
			&=\int_{\Rn} (\cR^{n-(j+1)}\zeta)(|x|)x\, D(\Hess v(x)[j],\Hess h_{B^n}(x)[n-(j+1)],I_n)\d x\\
			&\;\;\vdots\\
			&=\int_{\Rn} \zeta(|x|)x\, D(\Hess v(x)[j],I_n[n-j])\d x
		\end{align*}
		for every $v\in C^2(\Rn)$. Together with \eqref{eq:elem_symm_mixed_dis} this completes the proof.
	\end{proof}

	\subsection{Proof of Theorem~\ref{thm:main_existence} and Theorem~\ref{thm:main_rep_maj}}
	Let $j\in\{1,\ldots,n\}$ and $\zeta\in T_j^n$ be given, and set $\alpha=\binom{n}{j}\cR^{n-j}\zeta$. Since $T_j^j=T_n^n$, it follows from Lemma~\ref{le:r_trans_bij} that $\alpha\in T_n^n$. Thus, Proposition~\ref{prop:int_maj_is_a_vector-val} shows that
	\[
	\oz(v)=\int_{\Rn} \alpha(|x|)x\d\MA_j(v;x),\quad v\in\fconvf,
	\]
	defines a continuous, dually epi-translation invariant, rotation and $\On$ equivariant valuation {$\oz\colon\fconvf\to\Rn$}. Furthermore, by \eqref{eq:mixed_ma_dis} and Proposition~\ref{prop:int_swap} we have
	\[
		\oz(v)=\binom{n}{j} \int_{\Rn} (\cR^{n-j}\zeta)(|x|)x\, D(\Hess v(x)[j],\Hess h_{B^n}(x)[n-j])\d x =\int_{\Rn} \zeta(|x|)x\, [\Hess v(x)]_j \d x
	\]
	for every $v\in\fconvf\cap C^2(\Rn)$. Since $\fconvf\cap C^2(\Rn)$ is a dense subset of $\fconvf$, the valuation $\oz$ is uniquely determined by the last equality and thus, Theorem~\ref{thm:main_existence} follows. Together with \eqref{eq:r_l}, this also shows Theorem~\ref{thm:main_rep_maj}.
	\qed
	
	\subsection{Extended Representation Formulas}
	\label{se:extended_rep_formulas}
	Let $j\in\{1,\ldots,n\}$. By Theorem~\ref{thm:main_existence} and \eqref{eq:phi_j_hess} we have
	\begin{equation}
	\label{eq:t_j_zeta_hessian_c2}
	\ot_{j,\zeta}^*(v)=\int_{\Rn} \zeta(|x|)x\d\Phi_j(v;x)
	\end{equation}
	for every $v\in \fconvf\cap C^2(\Rn)$ and $\zeta\in T_j^n$. Furthermore, it follows from Theorem~\ref{thm:main_existence} and Proposition~\ref{prop:int_maj_is_a_vector-val} that \eqref{eq:t_j_zeta_hessian_c2} also holds whenever $v\in\fconvf$ and $\zeta\in T_n^n \subset T_j^n$.	Similar to \cite[Lemma 3.23]{colesanti_ludwig_mussnig_5}, we want to extend this representation, for general $\zeta\in T_j^n$, to
	\[
	\fconvfz = \{v\in \fconvf : v \text{ is of class } C^2 \text{ in a neighborhood of the origin}\}.
	\]
	
	\begin{lemma}
		\label{le:t_j_zeta_fconvz}
		Let $j\in\{1,\ldots,n\}$. If $\zeta\in T_j^n$, then
		\[
		\ot_{j,\zeta}^*(v)=\int_{\Rn} \zeta(|x|)x\d\Phi_j(v;x)
		\]
		for every $v\in\fconvfz$.
	\end{lemma}
	\begin{proof}
		Let $\zeta\in T_j^n$ and $v\in\fconvfz$ be given. By the definition of $\fconvfz$ there exists $r_0>0$ such that $v$ is of class $C^2$ on $r_0 B^n$. Using a standard mollification procedure, we can find a sequence $v_k\in\fconvf\cap C^2(\Rn)$, $k\in\N$, such that $v_k$ epi-converges to $v$ as $k\to\infty$ and such that we have convergence in the $C^2$-norm on $r_0 B^n$. Since $\zeta$ is continuous on $(0,\infty)$, it follows from the weak continuity of $w\mapsto \Phi_j(w;\cdot)$ that
		\begin{equation}
			\label{eq:lim_k_int_x_geq_r}
			\lim_{k\to\infty} \int_{\{x: |x|\geq r\}} \zeta(|x|)x \d\Phi_j(v_k;x)=\int_{\{x: |x|\geq r\}} \zeta(|x|)x \d\Phi_j(v;x)
		\end{equation}
		for every $r>0$ such that $\Phi_j(v;\{x: |x|=r\})=0$, which is satisfied for a.e.\ $r>0$ since $\Phi_j(v;\cdot)$ is locally finite. Furthermore, it follows from \eqref{eq:phi_j_hess} together with using polar coordinates that
		\begin{align*}
			\left|\int_{\{x: |x|<r\}} \zeta(|x|)x \d\Phi_j(v_k;x) \right| &= \left|\int_{\{x: |x|<r\}} \zeta(|x|)x\,[\Hess v_k(x)]_j \d x \right|\\
			&\leq \int_{\{x: |x|<r\}} |\zeta(|x|)|x| \,\left| [\Hess v_k(x)]_j\right| \d x\\
			&\leq n \kappa_n \sup_{|x|<r}\left| [\Hess v_k(x)]_j\right| \int_0^r |\zeta(s)|s^n \d s
		\end{align*}
		for every $k\in\N$ and $r>0$. Since $\lim_{s\to 0^+} \zeta(s)s^n=0$ and since $v_k$ converges to $v$ in the $C^2$-norm on $r_0 B^n$, this shows that for every $\varepsilon>0$ we can find $r_1\in(0,r_0)$ such that
		\[
		\left|\int_{\{x: |x|<r_1\}} \zeta(|x|)x \d\Phi_j(v_k;x) \right|<\varepsilon\quad\text{and}\quad \left|\int_{\{x: |x|<r_1\}} \zeta(|x|)x \d\Phi_j(v;x) \right|<\varepsilon
		\]
		for every $k\in\N$. Together with \eqref{eq:t_j_zeta_hessian_c2} and \eqref{eq:lim_k_int_x_geq_r} this gives
		\[
		\ot_{j,\zeta}^*(v)=\lim_{k\to\infty} \ot_{j,\zeta}^*(v_k) = 
		\lim_{k\to\infty} \int_{\Rn} \zeta(|x|)x\d\Phi_j(v_k;x) = \int_{\Rn} \zeta(|x|)x\d\Phi_j(v;x),
		\]
		where we have used the continuity of $\ot_{j,\zeta}^*$, due to Theorem~\ref{thm:main_existence}.
	\end{proof}
	
	\section{Further Representations}
	\label{se:further_reps}
	The purpose of this section is to find further representations of functional Minkowski vectors.
	First, in Section~\ref{se:kubota}, we will use recent results from \cite{hug_mussnig_ulivelli_support} to establish Kubota-type formulas for the operators $\ot_{j,\zeta}^*$, which also yields another explicit representation of functional Minkowski vectors, forgoing singular integrals. After this, we represent functional Minkowski vectors in terms of integrals with respect to area measures of $(n+1)$-dimensional convex bodies in Section~\ref{se:functions_to_bodies}. This will be needed in the context of smooth valuations in Section~\ref{se:smooth_vals}. As an intermediate step towards this representation, we first need a result similar to Proposition~\ref{prop:int_swap}, where one treats the case where integrals with respect to Hessian measures are transformed into integrals with respect to
	\[
	D(\Hess v(x)[j],\Hess v_{B}(x)[n-j])\d x,
	\]
	where $v_B\in \fconvf$ is given by
	\begin{equation}
	\label{eq:def_v_b}
	v_B(x)=h_{B^{n+1}}(x,-1)=\sqrt{1+|x|^2}
	\end{equation}
	for $x\in\Rn$. While the procedure for this is largely analogous to Section~\ref{se:existence}, there are also essential differences. In particular, the necessary transformation of the integral densities, which we introduce in Section~\ref{se:transform_t}, is not a bijection between different classes $T_j^n$ and $T_j^{n-l}$, but from $T_j^n$ onto itself.	
	
	\subsection{Kubota-type Formulas}
	\label{se:kubota}	
	The following Kubota-type formula for the measures $\MA_j(v;\cdot)$ is a consequence of
	\cite[Theorem 1.3]{hug_mussnig_ulivelli_support}. Here, for $j\in\{1,\ldots,n-1\}$ we denote by $\Grass{j}{n}$ the Grassmannian of $j$-dimensional linear subspaces of $\Rn$, on which we consider integration with respect to the Haar probability measure. For $E\in\Grass{j}{n}$ and $w\in\fconvfE$ we furthermore write $\MA_E(w;\cdot)$ for the Monge--Amp\`ere measure with respect to the ambient space $E$.
	\begin{theorem}
		\label{thm:kubota_fconvf}
		If $j\in\{1,\ldots,n-1\}$ and $\beta\colon \Rn\to [0,\infty)$ is measurable, then
		\[
		\frac{1}{\kappa_n}\int_{\Rn} \beta(x)\d\MA_j(v;x) = \frac{1}{\kappa_j} \int_{\Grass{j}{n}}\int_E \beta(x_E)\d\MA_E(v\vert_E;x_E)\d E
		\]
		for every $v\in\fconvf$.
	\end{theorem}

	We now establish Kubota-type formulas for functional intrinsic moments, where we emphasize that we consider $E\in\Grass{j}{n}$ as a subspace of $\Rn$.

	\begin{theorem}
		\label{thm:ot_j_zeta_kubota}
		If $j\in\{1,\ldots,n\}$ and $\zeta\in T_j^n$, then
		\[
		\ot_{j,\zeta}^*(v)=\int_{\Grass{j}{n}} \int_E \alpha(|x_E|)x_E \d\MA_E(v\vert_E;x_E)\d E
		\] 
		for every $v\in \fconvf$, where $\alpha\in T_n^n$ is given by
		\[
		\alpha(s)= \binom{n}{j}\frac{\kappa_n}{\kappa_j}\left(s^{n-j} \zeta(s)+(n-j)\int_s^\infty t^{n-j-1}\zeta(t) \d t\right)
		\]
		for $s>0$.
	\end{theorem}
	\begin{proof}
		Let $j\in\{1,\ldots,n\}$, $\zeta\in T_j^n$, and $\alpha\in T_n^n$ be given and fix some $i\in\{1,\ldots,n\}$. Let
		\[
		\beta(x)=\langle \alpha(|x|)x,e_i\rangle
		\]
		for $x\in\Rn$, which by Lemma~\ref{le:continuous_extension} defines a continuous function with compact support on $\Rn$. By Theorem~\ref{thm:main_rep_maj} and Theorem~\ref{thm:kubota_fconvf} we now have
		\[
		\langle \ot_{j,\zeta}^*(v),e_i\rangle = \int_{\Rn} \beta(x)\d\MA_j(v;x) = \int_{\Grass{j}{n}}\int_E \beta(x_E)\d\MA_E(v\vert_E;x)
		\]
		for every $v\in\fconvf$, which completes the proof.
	\end{proof}

	For $u\in\fconvs$ and $E\in\Grass{j}{n}$ we define the \textit{projection function} $\proj_E u\colon E\to\R$ as
	\[
	\proj_E u(x_E)=\min\nolimits_{z\in E^\perp} u(x_E+z),\quad x_E\in E,
	\]
	which is again a proper, lower semicontinuous, convex function. Since
	\[
	(\proj_E u)^*=u^*\vert_E,
	\]
	where on the left side the convex conjugate is taken with respect to the ambient space $E$ (see \cite[Theorem 11.23]{rockafellar_wets}), we immediately obtain the following representation for the dual operators $\ot_{j,\zeta}$ from Theorem~\ref{thm:ot_j_zeta_kubota} together with \eqref{eq:map_grad}. Here, we write $\nabla_E$ for the gradient with respect to the ambient space $E$.
	
	\begin{theorem}
		\label{thm:ot_j_zeta_kubota_fconvs}
		If $j\in\{1,\ldots,n\}$ and $\zeta\in T_j^n$, then
		\[
		\ot_{j,\zeta}(u)=\int_{\Grass{j}{n}} \int_{\dom(\proj_E u)} \alpha(|\nabla_E \proj_E u(x_E)|)\nabla_E \proj_E u(x_E) \d x_E \d E
		\] 
		for every $u\in \fconvs$, where $\alpha\in T_n^n$ is given by
		\[
		\alpha(s)= \binom{n}{j}\frac{\kappa_n}{\kappa_j}\left(s^{n-j} \zeta(s)+(n-j)\int_s^\infty t^{n-j-1}\zeta(t) \d t\right)
		\]
		for $s>0$.
	\end{theorem}
	
	\subsection{The Transform \texorpdfstring{$\boldsymbol{\cT}$}{T}}
	\label{se:transform_t}
	
	Instead of the transform $\cR$, the transform
	\[
	\cT \zeta(s)=\sqrt{1+s^2}\,\zeta(s) + \int_s^\infty \frac{t \zeta(t)}{\sqrt{1+t^2}}\d t
	\]
	with $\zeta\in C_b((0,\infty))$ and $s>0$ needs to be considered when integrals with respect to Hessian measures are transformed into integrals with respect to the measures $\MA(v[j],v_B[n-j];\cdot)$.	We start with a description of $\cT^l$ for $l\in\N$.
	
	\begin{lemma}
	\label{le:t_l}
	If $l\in\N$ and $\zeta\in C_b((0,\infty))$, then
	\[
	\cT^l \zeta(s)=\underbrace{(\cT\circ \cdots\circ\cT)}_l \zeta(s) = (1+s^2)^{\frac{l}{2}}\zeta(s)+l\int_s^\infty t(1+t^2)^{\frac{l}{2}-1} \zeta(t)\d t
	\]
	for every $s>0$.
	\end{lemma}
	\begin{proof}
		We prove the statement by induction on $l$. The statement trivially holds for $l=1$. For $l\geq 2$ and assuming that the statement holds for $l-1$, we now have
		\begin{align}
			\begin{split}
			\label{eq:calc_t_l_induction}
			\cT^l \zeta(s) &= \cT^{l-1}\cT \zeta(s)\\
			&= (1+s^2)^{\frac{l-1}{2}} \left(\sqrt{1+s^2}\zeta(s)+\int_s^\infty \frac{t \zeta(t)}{\sqrt{1+t^2}} \d t \right)\\
			&\quad+ (l-1) \int_s^\infty t(1+t^2)^{\frac{l-1}{2}-1} \left(\sqrt{1+t^2} \zeta(t) + \int_t^\infty \frac{r \zeta(r)}{\sqrt{1+r^2}} \d r \right) \d t\\
			&=(1+s^2)^{\frac{l}{2}} \zeta(s) + (l-1) \int_s^\infty t(t+t^2)^{\frac{l}{2}-1}\zeta(t)\d t\\
			&\quad + (1+s^2)^{\frac{l-1}{2}} \int_s^\infty \frac{t\zeta(t)}{\sqrt{1+t^2}}\d t + (l-1)\int_s^\infty  t (1+t^2)^{\frac{l-1}{2}-1} \int_t^\infty \frac{r \zeta(r)}{\sqrt{1+r^2}} \d r \d t
			\end{split}
		\end{align}
		for every $s>0$. Considering that $l\geq 2$ and that $\zeta$ has bounded support, integration by parts gives
		\begin{align*}
			(l-1) \int_s^\infty t (1+t^2)^{\frac{l-1}{2}-1} &\int_t^\infty \frac{r \zeta(r)}{\sqrt{1+r^2}} \d r \d t\\
			&= - (1+s^2)^{\frac{l-1}{2}} \int_s^\infty \frac{t \zeta(t)}{\sqrt{1+t^2}} \d t + \int_s^\infty (1+t^2)^{\frac{l}{2}-1} t \zeta(t) \d t
		\end{align*}
	for every $s>0$, which together with \eqref{eq:calc_t_l_induction} gives the desired result.
	\end{proof}

	Before we consider how the transform $\cT$ acts on the classes $T_j^n$, we need the following result.
	
	\begin{lemma}
		\label{le:limit_t_trans}
		Let $j\in\{1,\ldots,n\}$. If $\zeta\in T_j^n$, then
		\[
		\lim_{s\to 0^+} s^{n-j+1}\int_s^\infty t(1+t^2)^{\frac{m}{2}-1} \zeta(t) \d t =0
		\]
		for every $m\in\Z$.
	\end{lemma}
	\begin{proof}
		Let $\zeta\in T_j^n$ and $m\in\Z$ be given. If $\zeta$ is such that $\lim_{s\to 0^+}  \int_s^{\infty} t(1+t^2)^{\frac{m}{2}-1} \zeta(t) \d t$ exists and is finite, then the statement is trivial. In the remaining case, we use L'Hospital's rule together with the definition of $T_j^n$ to obtain
		\begin{align*}
		\lim_{s\to 0^+} \left| s^{n-j+1} \int_s^\infty t(1+t^2)^{\frac{m}{2}-1} \zeta(t)\d t \right| &\leq \lim_{s\to 0^+} \frac{\int_s^\infty |t(1+t^2)^{\frac{m}{2}-1} \zeta(t)|\d t}{\frac{1}{s^{n-j+1}}}\\
		&= \lim_{s\to 0^+} \frac{|s (1+s^2)^{\frac{m}{2}-1}\zeta(s)|}{\frac{n-j+1}{s^{n-j+2}}}\\
		&= \lim_{s\to 0^+} \frac{(1+s^2)^{\frac{m}{2}-1}}{n-j+1} |s^{n-j+3}\zeta(s)|\\
		&= 0.
		\end{align*}
	\end{proof}
	
	\begin{lemma}
		\label{le:prop_t_j_n}
		Let $j\in\{1,\ldots,n\}$ and $l\in\N$. The map $\cT^l\colon T_j^n\to T_j^n$ is a bijection with inverse
		\begin{equation}
		\label{eq:t_trans_inverse}
		\cT^{-l}\rho(s) = (\cT^{-1})^l \rho(s) = \frac{\rho(s)}{(1+s^2)^{\frac{l}{2}}} - l \int_s^\infty \frac{t \rho(t)}{(1+t^2)^{\frac{l}{2}+1}} \d t
		\end{equation}
		for $\rho\in T_j^n$ and $s>0$.
	\end{lemma}
	\begin{proof}
		Without loss of generality, let $l>0$. We start by showing that if $\zeta\in T_j^n$, then also $\cT^l\zeta\in T_j^n$. Note that it is enough to show this for the case $l=1$. First, it is easy to see that $\cT \zeta\in C_b((0,\infty))$. We now use the definition of $T_j^n$ together with Lemma~\ref{le:limit_t_trans} to obtain
		\[
			\lim_{s\to 0^+} s^{n-j+1} \cT \zeta(s) =\lim_{s\to 0^+} \sqrt{1+s^2} s^{n-j+1} \zeta(s) + \lim_{s\to 0^+} s^{n-j+1} \int_{s}^\infty \frac{t \zeta(t)}{\sqrt{1+t^2}}\d t = 0,
		\]
		which shows that $\cT \zeta\in T_j^n$.
		
		Next, we show that \eqref{eq:t_trans_inverse} is the inverse operation to $\cT^l$. For $\zeta\in T_j^n$ it follows from Lemma~\ref{le:t_l} that
		\begin{align}
			\begin{split}
			\label{eq:t_inv_t_l}
			\frac{\cT^l \zeta(s)}{(1+s^2)^{\frac ls}} - &l \int_s^\infty \frac{t \cT^l \zeta(t)}{(1+t^2)^{\frac{l}{2}+1}} \d t\\
			&= \zeta(s) + \frac{l}{(1+s^2)^{\frac{l}{2}}} \int_s^\infty t(1+t^2)^{\frac{l}{2}-1}\zeta(t)\d t\\
			&\quad - l \int_s^\infty \frac{t \zeta(t)}{1+t^2} \d t - l^2 \int_s^\infty \frac{t}{(1+t^2)^{\frac{l}{2}+1}}\int_t^\infty r (1+r^2)^{\frac{l}{2}-1} \zeta(r) \d r  \d t
			\end{split}
		\end{align}
		for every $s>0$. Taking into account that $\zeta$ has bounded support, integration by parts gives
		\begin{align*}
			l^2 \int_s^\infty \frac{t}{(1+t^2)^{\frac{l}{2}+1}}&\int_t^\infty r (1+r^2)^{\frac{l}{2}-1} \zeta(r) \d r  \d t\\
			&= l \frac{1}{(1+s^2)^\frac{l}{2}} \int_s^\infty t(1+t^2)^{\frac{l}{2}-1}\zeta(t)\d t-l\int_s^\infty \frac{t \zeta(t)}{1+t^2} \d t
		\end{align*}
		for $s>0$, which together with \eqref{eq:t_inv_t_l} shows that the left inverse of $\cT^l$ is given by \eqref{eq:t_trans_inverse}. Similarly, one shows that $\cT^l$ is the inverse of \eqref{eq:t_trans_inverse}.
		
		Lastly, let $\rho\in T_j^n$. We need to show that $\cT^{-l}\rho\in T_j^n$, where it is again sufficient to consider the case $l=1$. Clearly, $\cT^{-1}\rho\in C_b((0,\infty))$. Furthermore, by the definition of $T_j^n$ and Lemma~\ref{le:limit_t_trans},
		\[
		\lim_{s\to 0^+} s^{n-j+1}\cT^{-1}\rho(s)=\lim_{s\to 0^+}\frac{s^{n-j+1}\rho(s)}{\sqrt{1+s^2}} - \lim_{s\to 0^+} s^{n-j+1} \int_s^\infty \frac{t\rho(t)}{(1+t^2)^{\frac{3}{2}}}\d t = 0,
		\]
		which shows that $\cT^{-1}\rho\in T_j^n$ and completes the proof.
	\end{proof}	
	
	\subsection{From Convex Functions to Convex Bodies}
	\label{se:functions_to_bodies}	
	Having established the relevant properties of the transform $\cT$ in the previous subsection, we proceed analogously to Section~\ref{se:connecting_hessian_measures}. We omit the proofs as they are almost identical to those in Section~\ref{se:connecting_hessian_measures}, noting that the vector field	 
	\[
	F(x) = \left(\frac{\psi'(|x|)}{|x|}+\frac{1}{2} \psi'(|x|)|x| \right) x_i\, x + \left(\psi(|x|) + \frac{1}{2}\psi(|x|)|x|^2 + \int_{|x|}^\infty \psi(t) t \d t\right) e_i
	\]
	with $x\in\Rn\setminus\{o\}$ and $i\in\{1,\ldots,n\}$ needs to be used in the analog to Lemma~\ref{le:int_swap_transform}, which states that
	\begin{align*}
		\int_{\Rn} \psi(|x|)x\,&D(\Hess v_1(x),\ldots,\Hess v_k(x),I_n[n-k])\d x\\
		&=\int_{\Rn} (\cT \psi)(|x|) x\, D(\Hess v_1(x),\ldots,\Hess v_k(x),I_n(x)[n-k-1],\Hess v_{B}(x))\d x
	\end{align*}
	for $k\in\{1,\ldots,n-1\}$, $v_1,\ldots,v_k\in C^2(\Rn)$ such that $\Hess v_1\equiv 0$ in a neighborhood of the origin, and $\psi\in C_b^2((0,\infty))$. Eventually, we obtain the following result.
	
	\begin{proposition}
		\label{prop:int_swap_ball}
		Let $j\in\{1,\ldots,n-1\}$. If $\zeta\in T_j^n$, then
		\[
		\int_{\Rn} \zeta(|x|)x \, [\Hess v(x)]_j \d x = \binom{n}{j} \int_{\Rn} (\cT^{n-j}\zeta)(|x|) x\, D(\Hess v(x)[j],\Hess v_B(x)[n-j])\d x
		\]
		for every $v\in C^2(\Rn)$.
	\end{proposition}
	
	\begin{remark}
		Proposition~\ref{prop:int_swap_ball} shows that in many cases the functional Minkowski vectors $\ot_{j,\zeta}^*(v)$ with $\zeta\in T_j^n$ and $v\in\fconvf$ can also be represented as integrals with respect to the measures $\MA(v[j],v_{B}[n-j];\cdot)$, for example, when $v\in \fconvfz$ (cf. Lemma~\ref{le:t_j_zeta_fconvz}).
		However, in contrast to using the measures $\MA_j(v;\cdot)$, this alternative representation would still involve singular densities since $\cT^{n-j}\zeta\in T_j^n$. This is not surprising insofar as a representation of the family of functional intrinsic volumes that uses continuous density functions, without singularities, is essentially only possible with mixed Monge--Amp\`ere measures which arise from functions that are not differentiable at the origin (such as $h_{B^n}$). See \cite[Theorem 8.2]{colesanti_ludwig_mussnig_3}.
	\end{remark}

	Next, we connect mixed Monge--Amp\`ere measures of convex functions with mixed area measures of associated convex bodies in $\KN$. Here, the \textit{mixed area measure} $S(K_1,\ldots,K_n,\cdot)$ associated to $K_1,\ldots,K_n\in\KN$ is the unique Borel measure on $\sN$ that is symmetric in its entries such that
	\[
	S_n(\lambda_1 L_1 + \cdots+\lambda_m L_m,\cdot)=\sum_{i_1,\ldots,i_n=1}^m \lambda_{i_1}\cdots \lambda_{i_n} S(L_{i_1},\ldots,L_{i_n},\cdot)
	\]
	for every $L_1,\ldots,L_m\in\KN$, $\lambda_1,\ldots,\lambda_m\geq 0$, and $m\in\N$. The $j$th area measure of a given convex body $K\in\KN$ is obtained as the special case 
	\begin{equation}
	\label{eq:s_j_mixed_s}
	S_j(K,\cdot)=S(K[j],B^{n+1}[n-j],\cdot)
	\end{equation}
	for $j\in\{0,\ldots,n\}$. We furthermore use the \textit{gnomonic projection} $\gnom\colon \sN_-\to\Rn$,
	\[
	\gnom(z)=\frac{(z_1,\ldots,z_n)}{|z_{n+1}|}
	\]
	for $z=(z_1,\ldots,z_{n+1})\in\sN_-$, which is a $C^\infty$ diffeomorphism between the open lower half-sphere $\sN_-=\{z=(z_1,\ldots,z_{n+1})\in\sN : z_{n+1} <0 \}$ and $\Rn$. The result below is a reformulation of \cite[Corollary 4.9 and Lemma 4.10]{hug_mussnig_ulivelli_support}.
	
	\begin{proposition}
		\label{prop:ma_bodies_functions}
		For every measurable $\varphi\colon \Rn\to {[0,\infty)}$ with compact support and every $v_1,\ldots,v_n\in\fconvf$ there exist convex bodies $K^{v_1},\ldots,K^{v_n}\in\KN$ such that
		\[
		\int_{\Rn} \varphi(x)\d\MA(v_1,\ldots,v_n;x)= \int_{\sN_-}  \varphi(\gnom(z)) |z_{n+1}| \d S(K^{v_1},\ldots,K^{v_n};z).
		\]
		In addition, if for $j\in\{1,\ldots,n\}$ we have $v_j(x)=h_K(x,-1)$, $x\in\Rn$, for some $K\in\K^{n+1}$, then we may choose $K^{v_j}=K$.
	\end{proposition}
	
	We obtain the following alternative representation of functional Minkowski vectors with densities in $T_n^n$.
	
	\begin{theorem}
		\label{thm:bodies_fcts}
		Let $\alpha\in T_n^n$ and $j\in\{1,\ldots,n\}$. For every $v\in\fconvf$ there exists a convex body $K^v\in \KN$ such that for the $i$th coordinate of $\ot_{j,\alpha}^*(v)$, $i\in\{1,\ldots,n\}$, we have
		\[
		\big(\ot_{j,\alpha}^*(v)\big)_i = \binom{n}{j} \int_{\sN_-} (\cT^{n-j} \alpha)(|\gnom(z)|) z_i \d S_j(K^v,z). 
		\]
		In addition, if $v(x)=h_{K}(x,-1)$, $x\in\Rn$, for some $K\in\KN$, then we may choose $K^v=K$.
	\end{theorem}
	\begin{proof}
		By Theorem~\ref{thm:main_existence}, Proposition~\ref{prop:int_swap_ball}, and \eqref{eq:mixed_ma_dis} we have
		\begin{align}
			\begin{split}
			\label{eq:ot_j_alpha_ma_v_b}
			\big(\ot_{j,\alpha}^*(v)\big)_i &= \int_{\Rn} \alpha(|x|)x_i [\Hess v(x)]_j \d x\\
			&=\binom{n}{j} \int_{\Rn} (\cT^{n-j}\alpha)(|x|) x_i\, D(\Hess v(x)[j],\Hess v_B(x)[n-j])\d x\\
			&=\binom{n}{j} \int_{\Rn} (\cT^{n-j}\alpha)(|x|) x_i \d\MA(v[j],v_B[n-j];x)
			\end{split}
		\end{align}
		for every $v\in \fconvf\cap C^2(\Rn)$ and $i\in\{1,\ldots,n\}$. Since $\alpha\in T_n^n$ it follows from Lemma~\ref{le:prop_t_j_n} that also $\cT^{n-j}\alpha\in T_n^n$ and thus, by Lemma~\ref{le:continuous_extension}, the integrand $x\mapsto (\cT^{n-j}\alpha)(|x|) x_i$ is continuous with compact support on $\Rn$. Using the continuity of $\ot_{j,\alpha}^*$ and Proposition~\ref{prop:int_ma_is_a_val} we may therefore continuously extend \eqref{eq:ot_j_alpha_ma_v_b} from $\fconvf\cap C^2(\Rn)$ to $\fconvf$. The statement now follows from Proposition~\ref{prop:ma_bodies_functions} together with \eqref{eq:def_v_b} and \eqref{eq:s_j_mixed_s}.
	\end{proof}

	\begin{remark}
		\label{re:mistake_hadwiger4_measures}
		The observant reader might suspect at this point that Theorem~\ref{thm:bodies_fcts} can also be derived directly from \cite[Lemma 4.4]{colesanti_ludwig_mussnig_8}. However, the latter is not correct in the full generality with which it is stated in \cite{colesanti_ludwig_mussnig_8} and can, therefore, not be applied here.
	\end{remark}
	
	\section{Characterization}
	\label{se:characterization}
	
	\subsection{Retrieving the Densities}
	\label{section:retrieving_the_densities}
	Let $j\in\{1,\ldots,n\}$ and let $\zeta\in D_j^n$. As was shown in \cite[Lemma 2.15]{colesanti_ludwig_mussnig_5}, it is relatively straightforward to retrieve the density $\zeta$ from the functional intrinsic volume $\oV_{j,\zeta}^*$. This is done using the family of functions $v_s\in\fconvf$, $s\geq 0$, given by
	\[
	v_s(x)=\begin{cases}
		0\quad &\text{if } |x| \leq s,\\
		|x|-s &\text{if } |x|>s,
	\end{cases}
	\]
	and showing that $\oV_{j,\zeta}^*(v_s)$ equals a multiple of $\cR^{n-j}\zeta(s)$. Since $\cR^{n-j}$ is a bijection between $D_j^n$ and $C_c({[0,\infty)})$ (see \cite[Lemma 3.8]{colesanti_ludwig_mussnig_6}), this uniquely determines $\zeta$.
	
	In the following, we want to establish a similar result for the operators $\ot_{j,\zeta}^*$ with $\zeta\in T_j^n$. The difficulty here, however, is that these operators vanish on the radially symmetric functions $v_s$. We will, therefore, resort to a modification of the functions $v_s$, which entails considerably more complex calculations.	
	
	Let $n\geq 2$. For $s\geq 0$, let $w_s\in\fconvf$ be given by
	\[
	w_s(x)=\begin{cases}
		0\quad &\text{if } |x|\leq s \text{ and } x_n\geq 0,\\
		|x|-s\quad &\text{if } |x|> s \text{ and } x_n \geq 0,\\
		0\quad &\text{if } |(x_1,\ldots,x_{n-1})|\leq s \text{ and } x_n <0,\\
		|(x_1,\ldots,x_{n-1})|-s \quad &\text{if } |(x_1,\ldots,x_{n-1})|> s \text{ and } x_n < 0,
	\end{cases}
	\]
	where we write $x=(x_1,\ldots,x_n)$ for $x\in\Rn$. It is straightforward to check that $w_s$ is of class $C^2$ almost everywhere on $\Rn$, with the subdifferential at the exceptions for $s>0$ being
	\begin{equation}
	\label{eq:subdiff_w_s_1}
	\partial w_s(x) = \left\{r \frac{x}{|x|} : r \in [0,1] \right\}
	\end{equation}
	if $|x|=s$ and $x_n \geq 0$, and
	\begin{equation}
	\label{eq:subdiff_w_s_2}
	\partial w_s(x) = \left\{r \frac{(x_1,\ldots,x_{n-1},0)}{|(x_1,\ldots,x_{n-1})|} : r \in [0,1] \right\}
	\end{equation}
	when $|(x_1,\ldots,x_{n-1})|=s$ and $x_n <0$.
	
	We start with an auxiliary result, where we note that the case $j\in\{1,\ldots,n-1\}$ is a consequence of \cite[Lemma 2.15]{colesanti_ludwig_mussnig_1} and its proof. For the case $j=n$, we will use that
	\begin{equation}
	\label{eq:v_s_conjugate}
	v_s^*(x)=\ind_{B^n}(x)+s|x|
	\end{equation}
	for $x\in\Rn$ and $s\geq 0$, cf.\ \cite[(2.1) and (2.2)]{colesanti_ludwig_mussnig_1}.
	\begin{lemma}
		\label{le:phi_j_v_s}
		If $j\in\{1,\ldots,n\}$, then
		\[
		\Phi_j(v_s;s\sn)=\kappa_n \binom{n}{j}s^{n-j}
		\]
		for every $s\geq 0$.
	\end{lemma}
	\begin{proof}
		As explained above, the case $j\in\{1,\ldots,n-1\}$ has already been treated. It remains to consider the case $j=n$, in which it follows from \eqref{eq:psi_j_phi_j}, \eqref{eq:map_grad}, and \eqref{eq:v_s_conjugate} that
		\[
		\Phi_n(v_s;s\sn)=\Psi_n(v_s^*;s\sn)=\int_{B^n} \chi_{s\sn}(\nabla s |x|)\d x = \kappa_n,
		\]
		where $\chi_{s\sn}$ denotes the characteristic function of $s \sn$.
	\end{proof}
	
	In the following, we will also use that
	\begin{equation}
	\label{eq:cauchy_proj}
	\int_{\{z\in \sn : z_n \geq 0\}} z_n \d\hm^{n-1}(z) = \frac{1}{2} \int_{\sn} |\langle z,e_n\rangle| \d S_{n-1}(B^n,z) = V_{n-1}(\proj_{e_n^\perp}B^n)=\kappa_{n-1} 
	\end{equation}
	which is a consequence of Cauchy's projection formula (see \cite[(5.80)]{schneider_cb} or \cite[(A.45)]{gardner_gt}).
	
	\begin{proposition}
		\label{prop:retrieve_densities}
		Let $n\geq 2$ and let $j\in\{1,\ldots,n-1\}$. If $\zeta\in T_j^n$, then
		\[
		\ot_{j,\zeta}^*(w_s)=\frac{\kappa_{n-1}}{n} \binom{n}{j} s^{n-j+1}\zeta(s)\, e_n
		\]
		for every $s\geq 0$.
	\end{proposition}
	\begin{proof}
		Throughout the proof, we fix an arbitrary $s>0$, noting that it suffices to restrict to this case due to continuity. For every $\vartheta\in\On$ such that $\vartheta e_n=e_n$ it follows from the definition of $w_s$ together with the $\On$ equivariance of $\ot_{j,\zeta}^*$ that
		\[
		\ot_{j,\zeta}^*(w_s)=\ot_{j,\zeta}^*(w_s\circ \vartheta^{-1})=\vartheta \ot_{j,\zeta}^*(w_s).
		\]
		Thus, $\ot_{j,\zeta}^*(w_s)$ must be parallel to $e_n$, and we therefore focus on calculating the $n$th coordinate of this vector. Since $w_s\in\fconvfz$ it follows from Lemma~\ref{le:t_j_zeta_fconvz} that
		\[
		\big(\ot_{j,\zeta}^*(w_s)\big)_n = \int_{\Rn} \zeta(|x|)x_n\d\Phi_j(w_s;x).
		\]
		To calculate this integral, we split the domain $\Rn$ into several parts.
		\begin{itemize}
			\item Let $A_0=\{x \in\Rn : x_n =0\}$. It follows from \eqref{eq:subdiff_w_s_1} that
			\[
			\dim(\{x + t y : x\in A_0, y\in \partial w_s(s)\})<n
			\]
			for every $t\geq 0$. Thus, by \eqref{eq:def_phi_j}, we have $\Phi_j(w_s;\cdot)\equiv 0$ on $A_0$.
            
			\item Let $A_1= \{x \in\Rn : |x|<s , x_n > 0 \}\cup \{x\in\Rn : |(x_1,\ldots,x_{n-1})|<s , x_n <0\}$, which is an open subset of $\Rn$. Since $w_s\equiv 0$ on $A_1$, it follows from \eqref{eq:phi_j_hess} that $\Phi_j(w_s;\cdot)\equiv 0$ on $A_1$.
            
			\item The function $w_s$ is of class $C^2$ on $A_2 = \{x\in \Rn : |x|>s, x_n > 0\}$ and its Hessian matrix at any point $x$ in this set has $(n-1)$ eigenvalues equal to $1/|x|$ and the last eigenvalue equal to zero. Hence,
			\[
			[\Hess w_s(x)]_j = \binom{n-1}{j}\frac{1}{|x|^j}
			\]
			for every $x\in A_2$. By \eqref{eq:phi_j_hess} together with using polar coordinates and \eqref{eq:cauchy_proj} we now have
			\begin{align*}
				\int_{A_2} \zeta(|x|)x_n \d\Phi_j(w_s;x) &= \binom{n-1}{j} \int_{A_2} \frac{\zeta(|x|) x_n}{|x|^j}\d x\\
				&=\binom{n-1}{j} \int_s^\infty r^{n-j} \zeta(r)\d r \int_{\{z\in \sn : z_n>0\}} z_n \d \hm^{n-1}(z)\\
				&=\kappa_{n-1} \binom{n-1}{j}  \int_s^\infty r^{n-j}\zeta(r)\d r.
			\end{align*}
			
			\item The function $w_s$ is also of class $C^2$ on $A_3 = \{x\in\Rn : |(x_1,\ldots,x_{n-1})|>s, x_n< 0\}$ and the Hessian matrix $\Hess w_s(x)$ has $(n-2)$ eigenvalues equal to $1/|(x_1,\ldots,x_{n-1})|$ and two eigenvalues equal to zero on this set. Thus,
			\[
			[\Hess w_s(x)]_j = \begin{cases}
				\binom{n-2}{j} \frac{1}{|(x_1,\ldots,x_{n-1})|^j}\quad &\text{if } j\in\{1,\ldots,n-2\},\\
				0\quad &\text{if } j=n-1,
			\end{cases}
			\]
			for every $x\in A_3$. Therefore, by \eqref{eq:phi_j_hess} we trivially have
			\[
			\int_{A_3} \zeta(|x|)x_n \d\Phi_{n-1}(w_s;x)=0.
			\]
			Furthermore, for $j\in\{1,\ldots,n-2\}$ we use polar coordinates to obtain 
			\begin{align*}
				\int_{A_3} \zeta(|x|)x_n \d\Phi_j(w_s;x) &= \binom{n-2}{j} \int_{A_3} \frac{\zeta(|x|)x_n}{|(x_1,\ldots,x_{n-1})|^j} \d x\\
				&=-\binom{n-2}{j} \int_{\{y\in \R^{n-1}: |y|>s\}} \frac{1}{|y|^j}  \int_{0}^{\infty} \zeta(\sqrt{|y|^2+x_n^2})x_n \d x_n \d y\\
				&= -(n-1)\kappa_{n-1} \binom{n-2}{j}  \int_s^\infty r^{n-2-j} \int_0^\infty \zeta(\sqrt{r^2+x_n^2}) x_n \d x_n \d r\\
				&=-(n-1) \kappa_{n-1} \binom{n-2}{j}  \int_s^\infty r^{n-2-j} \int_r^\infty t \zeta(t) \d t \d r\\
				&=\kappa_{n-1} \binom{n-1}{j}\left(s^{n-1-j} \int_s^\infty t \zeta(t)\d t - \int_s^\infty r^{n-j} \zeta(r) \d r \right),
			\end{align*}
			where we have implicitly used that $n-1\geq 2$, together with the substitution $t=\sqrt{r^2+x_n^2}$ and integration by parts.
			
			\item Let $A_4=\{x\in \Rn : |x|=s, x_n> 0\}=s\sn\cap\{x\in \Rn : x_n> 0\}$. Since Hessian measures are locally determined (see \eqref{eq:loc_det}), it follows that $\Phi_j(w_s;\cdot)$ coincides with $\Phi_j(v_s;\cdot)$ on $\{x\in \Rn : x_n >0\}$ and, in particular, on $A_4$. Observe that $v_s$ is a radially symmetric function and therefore, $\Phi_j(v_s;\cdot)$ is invariant under rotations, which implies that $\Phi_j(v_s;\cdot)$, restricted to $s\sn$, must be the Haar measure on this set. This means that there exists a constant $c_{n,j,s}$ such that
			\[
			\int_{s \sn} b(z)\d\Phi_j(v_s;z)= c_{n,j,s}\int_{s\sn} b(z)\d\hm^{n-1}(z) 
			\]
			for every Borel measurable function $b$ on $s\sn$. By Lemma~\ref{le:phi_j_v_s} we have
			\[
			\kappa_n \binom{n}{j} s^{n-j} = \Phi_j(v_s; s\sn)= c_{n,j,s} \,\hm^{n-1}(s \sn) = c_{n,j,s}\, n \kappa_n s^{n-1},
			\] 
			and thus, $c_{n,j,s}=\frac{1}{n} \binom{n}{j}\frac{1}{s^{j-1}}$. Summing up, we obtain
			\begin{align*}
				\int_{A_4} \zeta(|x|)x_n \d\Phi_j(w_s;x)&=\int_{A_4} \zeta(|x|)x_n \d\Phi_j(v_s;x)\\
				&=\frac{1}{n}\binom{n}{j} \frac{1}{s^{j-1}}\int_{s\sn \cap \{x\in\Rn : x_n> 0\}} \zeta(|z|)z_n \d\hm^{n-1}(z)\\
				&=\frac{1}{n}\binom{n}{j}s^{n-j+1}\zeta(s)\int_{\{z\in\sn : z_n > 0\}} z_n \d\hm^{n-1}(z)\\
				&=\frac{\kappa_{n-1}}{n}\binom{n}{j} s^{n-j+1}\zeta(s)
			\end{align*}
			where we have used \eqref{eq:cauchy_proj}.
			
			\item Lastly, let $A_5=\{x\in\Rn : |(x_1,\ldots,x_{n-1})|=s, x_n <0\}$. Since
			\[
			w_s(x_1,\ldots,x_n)=v_s\vert_{\R^{n-1}}(x_1,\ldots,x_{n-1})
			\]
			for every $(x_1,\ldots,x_n)\in \{x\in\Rn : x_n < 0\}$, it follows from Lemma~\ref{le:phi_j_lower_dim} and the fact that Hessian measures are locally determined that
			\[
			\d\Phi_j(w_s;(x_1,\ldots,x_n)) = \d\Phi_j^{(n-1)}(v_s\vert_{\R^{n-1}};(x_1,\ldots,x_{n-1})) \d x_n
			\]
			on $A_5$. By Lemma~\ref{le:phi_j_v_s} we have $\Phi_j^{(n-1)}(v_s\vert_{\R^{n-1}};s\s^{n-2}) = \kappa_{n-1}\binom{n-1}{j}s^{n-1-j}$, and thus
			\begin{align*}
				\int_{A_5} \zeta(|x|)x_n \d\Phi_j(w_s;x)&=-\int_{0}^{\infty}\int_{s\s^{n-2}} \zeta(\sqrt{|y|^2+x_n^2})  x_n \d\Phi_j^{(n-1)}(v_s\vert_{\R^{n-1}};y)\d x_n\\
				&=-\kappa_{n-1}\binom{n-1}{j}s^{n-1-j}\int_0^\infty \zeta(\sqrt{s^2+x_n^2})x_n \d x_n\\
				&=-\kappa_{n-1}\binom{n-1}{j} s^{n-1-j}\int_s^\infty t \zeta(t) \d t. 
			\end{align*}
		\end{itemize}
		Since $\Rn$ is the disjoint union of the sets $A_0,\ldots,A_n$, we finally obtain
		\begin{align*}
			\int_{\Rn} &\zeta(|x|)x_n \d\Phi_j(w_s;x)\\
			&= \kappa_{n-1}\binom{n-1}{j}\int_s^\infty r^{n-j}\zeta(r)\d r + \kappa_{n-1}\binom{n-1}{j}\left(s^{n-1-j} \int_s^\infty t \zeta(t)\d t - \int_s^\infty r^{n-j} \zeta(r) \d r \right)\\
			&\quad+\frac{\kappa_{n-1}}{n}\binom{n}{j} s^{n-j+1}\zeta(s)-\kappa_{n-1}\binom{n-1}{j} s^{n-1-j}\int_s^\infty t \zeta(t) \d t\\
			&= \frac{\kappa_{n-1}}{n}\binom{n}{j} s^{n-j+1}\zeta(s),
		\end{align*}
		when $j\in\{1,\ldots,n-2\}$. Similarly, for $j=n-1$,
		\begin{align*}
			\int_{\Rn} \zeta(|x|)x_n \d\Phi_{n-1}(w_s;x) &=\kappa_{n-1}\binom{n-1}{n-1}\int_s^\infty r^{n-(n-1)}\zeta(r)\d r+\frac{\kappa_{n-1}}{n}\binom{n}{n-1} s^{n-(n-1)+1}\zeta(s)\\
			&\quad-\kappa_{n-1}\binom{n-1}{n-1} s^{n-1-(n-1)}\int_s^\infty t \zeta(t) \d t\\
			&=\kappa_{n-1} s^2 \zeta(s),
		\end{align*}
		which gives the desired result.
	\end{proof}
	
	To conclude this section, we show that the statement of Proposition~\ref{prop:retrieve_densities} also holds true for $j=n$, including the case $n=1$. For this, we first observe that
	\[
	w_s+s=h_{D^n}\vee s
	\]
	for every $s\geq 0$, where we write $D^n=B^n\cap \{x\in\Rn : x_n \geq 0\}$ and where we also use this as a definition of $w_s$ when $n=1$ (with $D^1=[0,1])$. It now follows from \cite[Lemma 2]{artstein_milman_annals_2009} and \eqref{eq:conj_support_indicator} that the convex conjugate $(w_s+s)^*\in\fconvs$ is the largest lower semicontinuous, convex function that is pointwise bounded from above by $\ind_{D^n}\wedge (\ind_{\{o\}}-s)$, or equivalently
	\begin{equation}
	\label{eq:w_s_conjugate}
	w_s^*(x) = \ind_{D^n}(x)+s|x|
	\end{equation}
	for every $x\in\Rn$ and $s\geq 0$.

	\begin{lemma}
		\label{le:retrieve_densities_n}
		If $\zeta\in T_n^n$, then
		\[
		\ot_{n,\zeta}^*(w_s)=\frac{\kappa_{n-1}}{n} s \zeta(s)\,e_n
		\]
		for every $s\geq 0$.
	\end{lemma}
	\begin{proof}
		Let $\zeta\in T_n^n$ be given. As in the proof of Proposition~\ref{prop:retrieve_densities}, it follows from the symmetries of $w_s$ and the $\On$ equivariance of $\ot_{n,\zeta}^*$, that $\ot_{n,\zeta}^*(w_s)$ is parallel to $e_n$. By \eqref{eq:map_grad} and \eqref{eq:w_s_conjugate} we now have
		\begin{align*}
			\big(\ot_{n,\zeta}^*(w_s)\big)_n=\big(\ot_{n,\zeta}(w_s^*)\big)_n&= \int_{\Rn} \zeta(|y|) y_n \d\MAp(w_s^*;y) \\
			&=\int_{D^n} \zeta(|\nabla w_s^*(x)|) \frac{\partial w_s^*}{\partial x_n}(x) \d x\\
			&= \int_{D^n} \zeta(s) \frac{s x_n}{|x|}\d x\\
			&= s\zeta(s) \int_{\{z\in\sn : z_n \geq 0\}} \int_0^1 z_n r^{n-1} \d r \d\hm^{n-1}(z)\\
			&= \frac{1}{n} s\zeta(s) \int_{\{z\in\sn : z_n \geq 0\}} z_n \d\hm^{n-1}(z)\\
			&=\frac{\kappa_{n-1}}{n} s \zeta(s)
		\end{align*}
		for every $s\geq 0$, where we have used \eqref{eq:cauchy_proj} in the last step.
	\end{proof}

	\begin{remark}
		Proposition~\ref{prop:retrieve_densities} and Lemma~\ref{le:retrieve_densities_n} together with the definition of the classes $T_j^n$ show that $t_{j,\zeta}^*(w_0)=o$ for every $j\in\{1,\ldots,n\}$ and $\zeta\in T_j^n$. This follows independently from the fact that $w_0=h_{D^n}$ and that $K\mapsto t_{j,\zeta}^*(h_K)$ defines a continuous, translation invariant, rotation equivariant valuation on $\Kn$, which by Proposition~\ref{prop:hadwiger_schneider} must be trivial.
	\end{remark}
	
	\subsection{The Abel Transform}
	\label{se:abel}
	To mimic Klain's approach to Hadwiger's theorem \cite{klain_1995}, we need to understand how the functional Minkowski vectors $\ot_{j,\zeta}$ behave on functions with lower dimensional domains. For this we consider the \textit{Abel transform} of a function $\zeta\in C_b((0,\infty))$, which is given by
	\[
	\cA \zeta(s)=\int_{-\infty}^\infty \zeta(\sqrt{s^2+t^2})\d t
	\]
	for $s>0$.
	Applying the Abel transform $k$ times, $k\in\N$, gives
	\begin{equation}
		\label{eq:abel_k_times}
		\cA^k \zeta(s)=\int_{\R^k} \zeta(\sqrt{|x|^2+s^2})\d x
	\end{equation}
	for $s>0$ and we remark that twofold application of the Abel transform generates a differentiable function (cf.\ \cite[Lemma 2.4]{colesanti_ludwig_mussnig_8}). For $\xi\in C_b^1((0,\infty))$, it is easy to verify that the inverse Abel transform is given by
	\begin{equation}
	\label{eq:inverse_abel}
	\cA^{-1}\xi(s)=-\frac{1}{\pi}\int_s^\infty \frac{\xi'(t)}{\sqrt{t^2-s^2}} \d t
	\end{equation}
	for $s>0$. See, for example, \cite[Chapter 13]{bracewell}. For the following result, we write $C_c^\infty({[0,\infty)})$ for the space of continuous functions with compact support on $[0,\infty)$ functions that are infinitely differentiable on $(0,\infty)$ and have one-sided derivatives of all orders at $0^+$.	
	
	\begin{lemma}
		\label{le:abel_smooth}
		Let $\tilde{\psi}\in C_c^\infty({[0,\infty)})$. If $\xi\in C_c^\infty({[0,\infty)})$ is such that
		\begin{equation}
		\label{eq:xi_psi_square}
		\xi(s)=\tilde{\psi}(s^2)
		\end{equation}
		for $s\geq 0$, then there exists $\psi\in C_c^\infty({[0,\infty)})$ such that
		\[
		\cA^{-1}\xi(s)=\psi(s^2)
		\]
		for $s\geq 0$.
	\end{lemma}
	\begin{proof}
		Let $\tilde{\psi}$ and $\xi$ be as in \eqref{eq:xi_psi_square}.	By \eqref{eq:inverse_abel} we have
		\[
		\cA^{-1}\xi(s) = -\frac{2}{\pi}\int_s^\infty \frac{\tilde{\psi}'(t^2)t}{\sqrt{t^2-s^2}} \d t = -\frac{2}{\pi}\int_0^\infty \tilde{\psi}'(s^2+r^2)\d r
		\]
		for $s>0$, where we have used the substitution $r=\sqrt{t^2-s^2}$. Since $\tilde{\psi}$ has compact support, this expression and all its derivatives continuously extend to $s=0$, completing the proof.
	\end{proof}

	\begin{lemma}
		\label{le:restriction_abel}
		Let $k\in\{1,\ldots,n-1\}$, $j\in\{1,\ldots,k\}$ and $\zeta\in C_c({[0,\infty)})$. If $u\in\fconvs$ is such that $\dom(u)\subseteq \R^k$, then
		\[
		\int_{\Rn} \zeta(|y|)y_i \d\Psi_j(u;y)=\int_{\R^k} (\cA^{n-k}\zeta)(|z|)z_i \d\Psi_j^{(k)}(u\vert_{\R^k};z)
		\]
		for every $i\in\{1,\ldots,k\}$ and
		\[
		\int_{\Rn} \zeta(|y|)y_i \d\Psi_j(u;y)=0
		\]
		for every $i\in\{k+1,\ldots,n\}$.
	\end{lemma}
	\begin{proof}
		If $u\in\fconvs$ is such that $\dom(u)\subseteq \R^k$, then there exists $w\in\fconvfk$ such that
		\[
		u^*(x_1,\ldots,x_n)=w(x_1,\ldots,x_k)
		\]
		for every $(x_1,\ldots,x_n)\in\Rn$. Thus, by \eqref{eq:psi_j_phi_j}, Lemma~\ref{le:phi_j_lower_dim} together with Fubini's theorem, and \eqref{eq:abel_k_times}, we now have
		\begin{align*}
			\int_{\Rn} \zeta(|y|)y_i \d\Psi_j(u;y) &= \int_{\Rn} \zeta(|x|)x_i \d\Phi_j(u^*;x)\\
			&= \int_{\R^k} \int_{\R}\cdots\int_{\R} \zeta(\sqrt{|z|^2+|(x_{k+1},\ldots,x_n)|^2})z_i \d x_{k+1} \cdots \d x_n \d\Phi_j^{(k)}(w;z)\\
			&=\int_{\R^k} (\cA^{n-k}\zeta)(|z|)z_i \d\Phi_j^{(k)}(w;z)\\
			&=\int_{\R^k} (\cA^{n-k}\zeta)(|z|)z_i \d\Psi_j^{(k)}(u\vert_{\R^k};z)
		\end{align*}
		for every $i\in\{1,\ldots,k\}$, where we have used that $w^*=u\vert_{\R^k}$ when considering convex conjugation with respect to the ambient space $\R^k$. Similarly, for $i\in\{k+1,\ldots,n\}$, we obtain
		\begin{align*}
			\int_{\Rn} \zeta(|y|)y_i \d\Psi_j(u;y) &=\int_{\R^k} \int_{\R}\cdots\int_{\R} \zeta(\sqrt{|z|^2+|(x_{k+1},\ldots,x_n)|^2})x_i \d x_{k+1} \cdots \d x_n \d\Phi_j^{(k)}(w;z)\\
			&=0,
		\end{align*}
		where we have used that $x_i \mapsto \zeta(\sqrt{|z|^2+|(x_{k+1},\ldots,x_n)|^2})x_i$ is an odd function on $\R$.
	\end{proof}
	
	\subsection{Classification of Smooth Valuations}
	\label{se:smooth_vals}
	
	Our proof of Theorem~\ref{thm:main_class} is based on the extended Klain approach developed in \cite{colesanti_ludwig_mussnig_4}, which in turn is based on Klain's strategy of reproving Hadwiger's theorem in \cite{klain_1995}. See also \cite{brauner_hofstaetter_ortega-moreno_zonal}. Based on a classification of simple valuations, Klain used an induction argument that involves extending valuations defined on lower-dimensional convex bodies to valuations defined on general convex bodies. As the results of Section~\ref{se:abel} show, such an extension is not possible in general for continuous functional Minkowski vectors since the relevant densities may not be in the image of the Abel transform (see \cite{faifman_hofstaetter} for a related problem). Thus, in order to prove Theorem~\ref{thm:main_class}, we will first establish a classification result under additional smoothness assumptions. Here, a continuous, dually epi-translation invariant valuation $\oz\colon\fconvf\to\Rn$ is called \textit{smooth} when each of the coordinates of
	\[
	K\mapsto \oz(h_K(\cdot,-1))
	\]
	defines a smooth valuation on $\KN$ (see \cite[Proposition 6.4]{knoerr_smooth}). A continuous, translation invariant valuation $\oZ\colon\KN\to\R$ is in turn \textit{smooth} when $\vartheta\mapsto \oZ\circ \vartheta$ is an infinitely differentiable map from $\operatorname{GL}(n)$ into the Banach space of continuous, translation invariant valuations on $\KN$ (we refer to \cite[Section 6.5]{schneider_cb} for an overview). For our purposes, the following simple result will be sufficient. Cf.\ \cite[Theorem 4.1]{schuster_wannerer_jems_2018} or \cite[(2.1)]{colesanti_ludwig_mussnig_4}.
	
	\begin{lemma}
		\label{le:smooth_val_kn}
		Let $j\in\{1,\ldots,n\}$ and let $\varphi\colon\sN\to\R$ be continuous. The valuation
		\[
		K\mapsto \int_{\sN} \varphi(z)\d S_j(K,z),\quad K\in\KN,
		\]
		is smooth, if and only if $\varphi$ is smooth.
	\end{lemma}
	
	For the announced classification result for smooth valuations, we first need to show that the relevant valuations are smooth. For this, we need the following consequence of \cite[Theorem 2]{whitney}.
	
	\begin{lemma}
		\label{le:whitney}
		Let $i\in\{1,\ldots,n\}$ and $\alpha\in T_n^n$. The map
		\[
		x\mapsto \alpha(|x|)x_i
		\]
		is of class $C^\infty$ on $\Rn$, if and only if there exists $\psi\in C_c^{\infty}({[0,\infty)})$ such that $\alpha(s)=\psi(s^2)$ for every $s>0$.
	\end{lemma}
	
	\begin{proposition}
		\label{prop:condition_smooth_val}
		Let $j\in\{1,\ldots,n\}$ and $\alpha\in T_n^n$. The valuation
		\begin{equation}
			\label{eq:v_mapsto_smooth_val}	
			v\mapsto \int_{\Rn} \alpha(|x|)x \d\Phi_j(v;x),\quad v\in\fconvf,
		\end{equation}
		is smooth, if and only if there exists $\psi\in C_c^{\infty}({[0,\infty)})$ such that $\alpha(s)=\psi(s^2)$ for every $s>0$.
	\end{proposition}
	\begin{proof}
		Throughout the proof, we consider the $i$th coordinate of \eqref{eq:v_mapsto_smooth_val} with some arbitrary, fixed $i\in\{1,\ldots,n\}$. If $v(x)=h_K(x,-1)$, $x\in\Rn$, with some $K\in\K^{n+1}$, then \eqref{eq:t_j_zeta_hessian_c2} and Theorem~\ref{thm:bodies_fcts} show that
		\[
			\int_{\Rn} \alpha(|x|)x_i \d\Phi_j(v;x) =\big(\ot_{j,\alpha}^*(v)\big)_i =\binom{n}{j}\int_{\sN_-} (\cT^{n-j} \alpha)(|\gnom(z)|) z_i \d S_j(K,z).
		\]
		Thus, by Lemma~\ref{le:smooth_val_kn}, the $i$th coordinate of \eqref{eq:v_mapsto_smooth_val} is smooth, if and only if
		\[
		\varphi(z)=(\cT^{n-j} \alpha)(|\gnom(z)|) z_i,\quad z\in \sN_-,
		\]
		is smooth, where we remark that we can trivially extend $\varphi$ from $\sN_-$ to $\sN$ with $\varphi(z)=0$ on $\sN\setminus \sN_-$ since $\alpha$ has bounded support. As the gnomonic projection and its inverse are smooth, smoothness of $\varphi$ is equivalent to smoothness of
		\[
		x\mapsto \varphi(\gnom^{-1}(x))= (\cT^{n-j}\alpha)(|x|)\sqrt{1+|x|^2} \,x_i
		\]
		on $\Rn$. By Lemma~\ref{le:whitney} together with the properties of $\cT$ and its inverse, this is equivalent to $\alpha(s)=\psi(s^2)$, $s>0$, for some $\psi\in C_c^{\infty}({[0,\infty)})$. 
	\end{proof}

	\begin{remark}
		We remark that the statement of Proposition~\ref{prop:condition_smooth_val} has previously been claimed in the more general \cite[Lemma 4.6]{colesanti_ludwig_mussnig_8}. However, the proof of the latter is erroneous, and it, therefore, cannot be used for the situation considered here. Cf.\ Remark~\ref{re:mistake_hadwiger4_measures}.
	\end{remark}
	
	Corresponding to the definition of smooth valuations on $\fconvf$, we call a continuous, epi-translation invariant valuation $\oz\colon\fconvs\to\Rn$ \textit{smooth} when
	\[
	v\mapsto \oz(v^*),\quad v\in\fconvf,
	\]
	defines a smooth valuation. We, therefore, immediately obtain the following dual version of Proposition~\ref{prop:condition_smooth_val}.
	
	\begin{proposition}
		\label{prop:condition_smooth_val_fconvs}
		Let $j\in\{1,\ldots,n\}$ and $\alpha\in T_n^n$. The valuation
		\begin{equation*}	
			u\mapsto \int_{\Rn} \alpha(|y|)y \d\Psi_j(u;y),\quad u\in\fconvs,
		\end{equation*}
		is smooth, if and only if there exists $\psi\in C_c^{\infty}({[0,\infty)})$ such that $\alpha(s)=\psi(s^2)$ for every $s>0$.
	\end{proposition}

	We now treat the one-dimensional case of the main result of this section.
	
	\begin{proposition}
		\label{prop:class_smooth_vals_dim_1}
		A map $\oz\colon\fconvsO\to\R$ is a smooth, epi-translation invariant, reflection equivariant valuation, if and only if there exists $\psi\in C_c^{\infty}([0,\infty))$ such that
		\begin{equation}
		\label{eq:smooth_val_one_dim}
		\oz(u)=\int_{\R} \psi(|y|^2)y\d\Psi_1(u;y)
		\end{equation}
		for every $u\in\fconvsO$.
	\end{proposition}
	\begin{proof}
		It follows from Proposition~\ref{prop:int_maj_is_a_vector-val_fconvs} that \eqref{eq:smooth_val_one_dim} defines a continuous, epi-translation invariant, reflection equivariant valuation. Furthermore, by Proposition~\ref{prop:condition_smooth_val_fconvs}, the valuation is also smooth.
		
		Conversely, if $\oz$ is a smooth, epi-translation invariant, reflection equivariant valuation, then it follows from Theorem~\ref{thm:mcmullen_fconvs} that $\oz$ can be decomposed in epi-homogeneous terms with degrees $0$ and $1$. Theorem~\ref{thm:class_0_hom} and Theorem~\ref{thm:class_n_hom}, together with Proposition~\ref{prop:condition_smooth_val_fconvs}, now show that $\oz$ must be as in \eqref{eq:smooth_val_one_dim}.
	\end{proof}

	The following result is a reformulation of \cite[Lemma 8.3]{colesanti_ludwig_mussnig_3}, where it was obtained as a direct consequence of the functional Hadwiger theorem together with the properties of the transform $\cR$. In fact, Proposition~\ref{prop:oz_determined_by_ut} can be regarded as equivalent to the functional Hadwiger theorem, and we note that a new proof of this statement was recently presented in \cite[Section 5]{knoerr_zonal}. Here, we use the family of functions $u_s\in\fconvs$, $s\geq 0$, defined as $u_s(x)=v_s^*(x)=\ind_{B^n} + s |x|$ for $x\in\Rn$ (cf.\ \eqref{eq:v_s_conjugate}).
	\begin{proposition}
		\label{prop:oz_determined_by_ut}
		Let $n\geq 2$ and $j\in\{0,\ldots,n\}$. If $\oZ\colon\fconvs\to\R$ is a continuous, epi-translation and rotation invariant valuation that is epi-homogeneous of degree $j$ such that $\oZ(u_s)=0$ for every $s\geq 0$, then $\oZ\equiv 0$.
	\end{proposition}

	Throughout the remainder of this section we embed $\R^{n-1}$ into $\Rn$ via the identification $\R^{n-1}=e_n^\perp$ and consider $\fconvsm$ as the set of all $u\in\fconvs$ such that $\dom(u)\subseteq \R^{n-1}$. 
	\begin{lemma}
		\label{le:restrict_oz}
		Let $n\geq 3$. If $\oz\colon\fconvs\to\Rn$ is a continuous, epi-translation invariant, rotation equivariant valuation, then $\oz(u)\in\R^{n-1}$ for every $u\in\fconvsm$. For $n=2$, the statement holds if we replace rotation equivariance with $\Ot$ equivariance.
	\end{lemma}
	\begin{proof}
	 	Let $\oz\colon\fconvs\to\Rn$ be given and define $\tilde{\oZ}\colon\fconvsm\to\R$ as the $n$th coordinate of the restriction of $\oz$ to $\fconvsm$. That is,
		\[
		\tilde{\oZ}(u)=(\oz\vert_{\fconvsm}(u))_n
		\]
		for $u\in\fconvsm$. We need to show that $\tilde{\oZ}$ vanishes identically.
		
		We start with the case $n\geq 3$. By Theorem~\ref{thm:mcmullen_fconvs}, applied coordinate-wise, we may assume without loss of generality that $\oz$ is epi-homogeneous of degree $j\in\{0,\ldots,n\}$. If $j=n$, then Theorem~\ref{thm:class_n_hom} immediately shows that $\oz$ vanishes on functions with lower dimensional domains and thus $\tilde{\oZ}\equiv 0$. For $j<n$, we observe that the properties of $\oz$ imply that also $\tilde{\oZ}$ is a continuous, epi-translation invariant valuation that is epi-homogeneous of degree $j$. Now for $\vartheta\in \SOnm$ let $\varTheta\in\SOn$ be given by
		\[
		\varTheta=\begin{pmatrix}
			\vartheta & 0\\
			0 & 1
		\end{pmatrix}.
		\]
		The rotation equivariance of $\oz$ shows that
		\[
		\tilde{\oZ}(u\circ \vartheta^{-1})=(\oz\vert_{\fconvsm}(u\circ \varTheta^{-1}))_n = (\varTheta \oz\vert_{\fconvsm}(u))_n = (\oz\vert_{\fconvsm}(u))_n = \tilde{\oZ}(u)
		\]
		for every $u\in\fconvsm$ and, therefore, $\tilde{\oZ}$ is also rotation invariant.
		Next, for $s\geq 0$ let $\tilde{u}_s\in\fconvsm$ be given by $\tilde{u}_s(x)=\ind_{B^{n-1}}(x)+s|x|$. For $\eta=\operatorname{diag}(1,\ldots,1,-1,-1)\in\SOn$ we have
		\begin{align*}
		\tilde{\oZ}(\tilde{u}_s)&=(\oz\vert_{\fconvsm}(u_s))_n\\
		&= (\oz\vert_{\fconvsm}(u_s\circ \eta^{-1}))_n\\
		&= (\eta \oz\vert_{\fconvsm}(u_s))_n\\
		&= -(\oz\vert_{\fconvsm}(u_s))_n\\
		&= -\tilde{\oZ}(u_s)
		\end{align*}
		for every $s\geq 0$. Proposition~\ref{prop:oz_determined_by_ut}, applied with respect to an ambient space of dimension $n-1\geq 2$, now shows that $\tilde{\oZ}(u)=0$ for every $u\in\fconvsm$.
		
		Lastly, when $n=2$ and $\oz$ is $\Ot$ equivariant, we use $\varTheta=\operatorname{diag}(1,-1)\in\Ot$ to obtain
		\[
		\tilde{\oZ}(u)=(\oz\vert_{\fconvsO}(u\circ \varTheta^{-1}))_2 = (\varTheta \oz\vert_{\fconvsO}(u))_2 = -(\oz\vert_{\fconvsO}(u))_2 = -\tilde{\oZ}(u)
		\]
		for every $u\in\fconvsO$, which shows that $\tilde{\oZ}\equiv 0$.
	\end{proof}

	\begin{theorem}
		For $n\geq 3$, a map $\oz\colon \fconvs\to\Rn$ is a smooth, epi-translation invariant, rotation equivariant valuation, if and only if there exist functions $\psi_1,\ldots,\psi_n\in C_c^{\infty}([0,\infty))$ such that
		\begin{equation}
		\label{eq:class_oz_smooth_fconvs}
		\oz(u)=\sum_{j=1}^n \int_{\Rn} \psi_j(|y|^2)y \d\Psi_j(u;y)
		\end{equation}
		for every $u\in\fconvs$. For $n\leq 2$, the same representation holds if we replace rotation equivariance with $\On$ equivariance.
	\end{theorem}
	\begin{proof}
		If $\psi_1,\ldots,\psi_n\in C_c^{\infty}([0,\infty))$, then it follows from Proposition~\ref{prop:int_maj_is_a_vector-val_fconvs} together with Proposition~\ref{prop:condition_smooth_val_fconvs} that \eqref{eq:class_oz_smooth_fconvs} is a smooth, epi-translation invariant, $\On$ equivariant valuation.
		
		Conversely, let $\oz\colon\fconvs\to\R$ be a smooth, epi-translation invariant, rotation equivariant (or $\On$ equivariant for $n\leq 2$) valuation. We use induction on the dimension $n$ to show that $\oz$ is as in \eqref{eq:class_oz_smooth_fconvs}. For the base case $n=1$, the statement follows from Proposition~\ref{prop:class_smooth_vals_dim_1}. Now let $n\geq 2$ and assume that the statement is true for the case $n-1$. We denote by $\tilde{\oz}$ the restriction of $\oz$ to $\fconvsm$ and note that Lemma~\ref{le:restrict_oz} shows that $\tilde{\oz}$ takes values in $\R^{n-1}$. Clearly, $\tilde{\oz}$ is a smooth, epi-translation invariant valuation on $\fconvsm$. Now for $\vartheta\in\Om$ let $\varTheta\in\SOn$ be given by
		\[
		\varTheta=\begin{pmatrix}
			\vartheta & 0\\
			0 & \det(\vartheta)
		\end{pmatrix}.
		\]
		The rotation equivariance of $\oz$ shows that
		\[
		\tilde{\oz}(u\circ \vartheta^{-1})=\oz\vert_{\fconvsm}(u\circ \varTheta^{-1})=\varTheta \oz\vert_{\fconvsm}(u) = \vartheta \tilde{\oz}(u),
		\]
		for every $u\in\fconvsm$, where we have used that the $n$th coordinate of $\oz\vert_{\fconvsm}(u)$ vanishes. Thus, $\tilde{\oz}$ is also $\Om$ equivariant (and in particular $\SOnm$ equivariant for $n-1\geq 3$) and it follows from the induction assumption that there exist functions $\tilde{\psi}_1,\ldots,\tilde{\psi}_{n-1}\in C_c^{\infty}({[0,\infty)})$ such that
		\[
		\tilde{\oz}(u)=\sum_{j=1}^{n-1} \int_{\R^{n-1}} \tilde{\psi}_j(|y|^2)y \d\Psi_j^{(n-1)}(u;y)
		\]
		for every $u\in\fconvsm$. By Lemma~\ref{le:abel_smooth} and Lemma~\ref{le:restriction_abel} there exist functions $\psi_j\in C_c^{\infty}({[0,\infty)})$, $j\in\{1,\ldots,n-1\}$, such that
		\begin{equation}
		\label{eq:oz_induction_step}
		\oz(u)=\sum_{j=1}^{n-1}\int_{\Rn} \psi_j(|y|^2)y \d\Psi_j^{(n)}(u;y)
		\end{equation}
		for every $u\in\fconvs$ such that $\dom(u)\subseteq\R^{n-1}$.
		Now let $\bar{\oz}\colon\fconvs\to\Rn$ be given by
		\begin{equation}
		\label{eq:bar_oz}
		\bar{\oz}(u)=\oz(u)- \sum_{j=1}^{n-1}\int_{\Rn} \psi_j(|y|^2)y \d\Psi_j^{(n)}(u;y)
		\end{equation}
		for $u\in\fconvs$. It follows from the assumptions on $\oz$ together with the first part of the proof that $\bar{\oz}$ is a smooth, epi-translation invariant, rotation equivariant (or $\Ot$ equivariant when $n=2$) valuation. Moreover, if $u\in\fconvs$ is such that $\dom(u)\subseteq E$ for some $(n-1)$-dimensional linear subspace $E\subset \Rn$, then we can find a rotation $\vartheta_E\in\SOn$ such that $\dom(u\circ \vartheta_E^{-1})\subseteq \R^{n-1}$. The rotation equivariance of $\bar{\oz}$ and \eqref{eq:oz_induction_step} now show that
		\[
		\bar{\oz}(u) = \vartheta_E^{-1} \,\bar{\oz}(u\circ \vartheta_E^{-1})=o.
		\]
		Together with the epi-translation invariance of $\bar{\oz}$, this implies that $\bar{\oz}$ is simple. Thus, it follows from Theorem~\ref{thm:class_simple}, and Proposition~\ref{prop:condition_smooth_val_fconvs} that there exists $\psi_n\in C_c^{\infty}({[0,\infty)})$ such that
		\[
		\bar{\oz}(u)=\int_{\Rn}\psi_n(|y|^2)y\d\Psi_n^{(n)}(u;y)
		\]
		for every $u\in\fconvs$. Together with \eqref{eq:bar_oz}, this gives the desired representation \eqref{eq:class_oz_smooth_fconvs}.
	\end{proof}
	
	We close this section with the equivalent formulation of the last result for smooth valuations on $\fconvf$.
	
	\begin{theorem}
		\label{thm:class_smooth}
		For $n\geq 3$, a map $\oz\colon \fconvf\to\Rn$ is a smooth, dually epi-translation invariant, rotation equivariant valuation, if and only if there exist functions $\psi_1,\ldots,\psi_n\in C_c^{\infty}([0,\infty))$ such that
		\begin{equation*}
			\oz(v)=\sum_{j=1}^n \int_{\Rn} \psi_j(|x|^2)x \d\Phi_j(v;x)
		\end{equation*}
		for every $v\in\fconvf$. For $n\leq 2$, the same representation holds if we replace rotation equivariance with $\On$ equivariance.
	\end{theorem}
	
	\subsection{Approximation and Support}
	The following result, which shows that continuous valuations can be approximated by smooth valuations, is a consequence of \cite[Theorem 1]{knoerr_smooth}.
	\begin{lemma}
		\label{le:approx_smooth}
		If $\oz\colon\fconvf\to\Rn$ is a continuous, dually epi-translation invariant valuation that is homogeneous of degree $j$, then there exists a sequence of smooth, dually epi-translation invariant valuations $\oz_k\colon\fconvf\to\Rn$, $k\in\N$, that converges to $\oz$ and such that each of the $\oz_k$ is homogeneous of degree $j$.
	\end{lemma}
	\noindent
	Here, the relevant spaces of valuations are equipped with the topology of uniform convergence on compact subsets, for which we have the following description \cite[Proposition 2.4]{knoerr_support}.
	
	\begin{lemma}
		\label{le:compact_subsets}
		A set $V\subset \fconvf$ is relatively compact, if and only if for every compact subset $A\subset\Rn$ there exists a constant $c_A$ such that
		\[
		\sup_{x\in A} |v(x)|\leq c_A
		\]
		for every $v\in V$.
	\end{lemma}

	A version of Lemma~\ref{le:approx_smooth} is known in which the approximating sequence for rotation invariant $\oz$ can also be selected as rotation invariant \cite[Proposition 6.6]{knoerr_smooth}. For our purposes, however, we need a variant for rotation equivariant valuations.
	
	\begin{lemma}
		Let $\oz$ and $\oz_k$, $k\in\N$, be as in Lemma~\ref{le:approx_smooth}. If $\oz$ is rotation equivariant, then also the valuations $\oz_k$, $k\in\N$, can be chosen as rotation equivariant. Similarly, if $\oz_k$ is $\On$ equivariant, then also $\oz_k$, $k\in\N$, can be chosen to have this property.
	\end{lemma}
	\begin{proof}
		Let a continuous, dually epi-translation invariant, rotation equivariant valuation $\oz$ on $\fconvf$ be given that is homogeneous of degree $j$, and let $\oz_k$, $k\in\N$, be an approximating sequence of smooth valuations as in Lemma~\ref{le:approx_smooth}. Define $\bar{\oz}_k\colon\fconvf\to\Rn$ as
		\[
		\bar{\oz}_k(v)=\int_{\SOn} \vartheta \oz_k(v\circ \vartheta) \d\vartheta
		\]
		for $v\in\fconvf$ and $k\in\N$, where we integrate with respect to the Haar probability measure on $\SOn$. Clearly, each $\bar{\oz}_k$ is a smooth, epi-translation invariant valuation, where the proof of smoothness is analogous to \cite[Proposition 6.6]{knoerr_smooth}. Moreover,
		\[
		\bar{\oz}_k(v\circ \eta^{-1}) = \int_{\SOn} \vartheta \oz_k(v\circ \eta^{-1}\circ \vartheta) \d\vartheta = \int_{\SOn} \eta \,\varphi \oz_k(v\circ \varphi)\d\varphi = \eta\, \bar{\oz}_k(v)
		\]
		for every $\eta\in\SOn$ and $v\in\fconvf$, where we used $\varphi=\eta^{-1}\vartheta$. Thus, the valuations $\bar{\oz}_k$, $k\in\N$, are also rotation equivariant.
		
		It remains to show that $\bar{\oz}_k$ converges to $\oz$ as $k\to\infty$. For this, let $V$ be a compact subset of $\fconvf$ and let $\varepsilon>0$ be arbitrary. It easily follows from Lemma~\ref{le:compact_subsets} that also the set $\bar{V} = \{v\circ \vartheta : v\in V , \vartheta\in \SOn\}$ is compact. Thus, by the properties of the sequence $\oz_k$, there exists $k_0\in\N$ such that
		\[
		|\oz_k(v)-\oz(v)|<\varepsilon
		\]
		for every $v\in \bar{V}$ and $k\geq k_0$. Using that $\int_{\SOn} \vartheta \oz(v\circ \vartheta)\d\vartheta = \oz(v)$ as $\oz$ is rotation equivariant, we now have
		\begin{align*}
			|\bar{\oz}_k(v)-\oz(v)| &= \left|\int_{\SOn} \vartheta \oz_k(v\circ \vartheta) \d\vartheta - \int_{\SOn} \vartheta \oz(v\circ \vartheta)\d\vartheta \right|\\
			&\leq \int_{\SOn} \left|\vartheta\left(\oz_k(v\circ \vartheta)-\oz(v\circ \vartheta)\right)  \right| \d\vartheta\\
			&=\int_{\SOn} \left|\oz_k(v\circ \vartheta)-\oz(v\circ \vartheta) \right| \d\vartheta\\
			&< \varepsilon
		\end{align*}
		for every $v\in V$ and $k\geq k_0$, which shows that $\bar{\oz}_k$ converges uniformly to $\oz$ on compact subsets.
		
		Lastly, we remark that the proof for $\On$ equivariant valuations is analogous to the above.
	\end{proof}
	
	We continue with the \textit{support} $\supp(\oz)$ of a continuous, dually epi-translation invariant valuation $\oz\colon\fconvf\to\Rn$, for which we have the following description due to \cite[Proposition 6.3]{knoerr_support} and where we remark that each such valuation has compact support \cite[Corollary 6.2]{knoerr_support}.
	
	\begin{lemma}
		\label{le:description_support}
		Let $\oz\colon\fconvf\to\Rn$ be a continuous, dually epi-translation invariant valuation. The support of $\oz$ is minimal (with respect to inclusion) among all closed sets $A\subset\Rn$ with the following property: if $v,w\in\fconvf$ are such that $v=w$ on an open neighborhood of $A$, then $\oz(v)=\oz(w)$.
	\end{lemma}

	We give a simple proof of the following result, which can be independently derived from \cite[Corollary 2.11]{knoerr_unitarily}.

	\begin{lemma}
		\label{le:supp_ot_supp_zeta}
		Let $j\in\{1,\ldots,n\}$ and $\zeta\in T_j^n$. If $R=\max\{|x|: x\in \supp(\ot_{j,\zeta}^*)\}$, then $\supp(\zeta)\subseteq [0,R]$.
	\end{lemma}
	\begin{proof}
		Assume that there exists $\bar{s}>R$ such that $\zeta(\bar{s})\neq 0$. By definition, the function $w_{\bar{s}}$ coincides with the constant zero function $\mathbf{0}\in\fconvf$ on a neighborhood of $R B^n$. However, $\ot_{j,\zeta}^*(\mathbf{0})=o$, while Proposition~\ref{prop:retrieve_densities} and the choice of $\bar{s}$ show that $\ot_{j,\zeta}^*(w_{\bar{s}})\neq o$. By Lemma~\ref{le:description_support}, this is a contradiction.
	\end{proof}
	
	Our interest in the support of the valuations $\ot_{j,\zeta}^*$ stems from the following result, which was shown in \cite[Proposition 1.2]{knoerr_support}.
	
	\begin{proposition}
		\label{prop:compact_supports_sequence}
		If $\oz_k\colon\fconvf\to\Rn$, $k\in\N$, is a sequence of continuous, dually epi-translation invariant valuations that converges to a continuous, dually epi-translation invariant valuation $\oz\colon\fconvf\to\Rn$, then there exists a compact set $A\subset \Rn$ such that the supports of $\oz$ and $\oz_k$, $k\in\N$, are contained in $A$.
	\end{proposition}
	
	\subsection{Proof of Theorem~\ref{thm:main_class}}
	For given $\zeta_1\in T_1^n,\ldots,\zeta_n\in T_n^n$ it follows from Theorem~\ref{thm:main_existence} that $\oz=\sum_{j=1}^n \ot_{j,\zeta_j}
	^*$ defines a continuous, dually epi-translation invariant, rotation equivariant valuation. In particular, this valuation is also $\On$ equivariant.
	
	Now let a continuous, dually epi-translation invariant, rotation equivariant (or $\On$ equivariant when $n\leq 2$) valuation $\oz\colon\fconvf\to\Rn$ be given. Considering Theorem~\ref{thm:mcmullen_fconvf} we may assume that $\oz$ is homogeneous of degree $j\in\{0,1,\ldots,n\}$. By Theorem~\ref{thm:class_0_hom} and Theorem~\ref{thm:class_n_hom} together with duality it is enough to consider $j\in\{1,\ldots,n-1\}$. Using Lemma~\ref{le:approx_smooth}, we can find a sequence of smooth, dually epi-translation invariant, rotation equivariant (or $\On$ equivariant) valuations $\oz_k\colon \fconvf\to\Rn$, $k\in\N$, each homogeneous of degree $j$, that converge to $\oz$. Theorem~\ref{thm:class_smooth} shows that for each $k\in\N$ there exists $\zeta_k\in C_c^{\infty}({[0,\infty)})$ (of the form $\zeta_k(t)=\psi_k(t^2)$) such that
	\[
	\oz_k(v)=\int_{\Rn} \zeta_k(|x|)x \d\Phi_j(v;x)=\ot_{j,\zeta_k}^*(v)
	\]
	for every $v\in\fconvf$, where the second equation follows from \eqref{eq:t_j_zeta_hessian_c2}. Since trivially $\zeta_k\in T_j^n$, it follows from	Proposition~\ref{prop:retrieve_densities} that
	\[
	\big(\oz_k(w_s)\big)_n=\binom{n}{j}\frac{\kappa_{n-1}}{n} \beta_k(s)
	\]
	for every $s\geq 0$ and $k\in\N$, where $\beta_k\in C_c({[0,\infty)})$ is such that $\beta_k(s)=s^{n-j+1}\zeta_k(s)$ for $s>0$ and $\beta_k(0)=0$. In particular, $\big(\oz_k(w_0)\big)_n=0$. Let now $\beta\colon {[0,\infty)}\to\R$ be such that
	\[
	\big(\oz(w_s)\big)_n = \binom{n}{j}\frac{\kappa_{n-1}}{n} \beta(s)
	\]
	for $s\geq 0$. For any compact set $C\subset {[0,\infty)}$ it follows from Lemma~\ref{le:compact_subsets} that the sequence $\oz_k$ converges uniformly to $\oz$ on $\{w_s : s\in C\}$ as $k\to\infty$. This implies that the functions $\beta_k(s)$ converge uniformly to $\beta$ on any such set $C$, which implies that $\beta\in C({[0,\infty)})$ and that $\beta(0)=0$. We remark that these properties can also be derived from the properties of $\oz$. Furthermore, Lemma~\ref{le:supp_ot_supp_zeta} and Proposition~\ref{prop:compact_supports_sequence} show that there exists $R>0$ such that $\supp(\zeta_k),\supp(\beta_k)\subseteq [0,R]$ for every $k\in\N$. Thus, we conclude that also $\beta\in C_c({[0,\infty)})$ and furthermore, that $\beta_k$ converges uniformly to $\beta$ on $[0,\infty)$.
	
	We now set $\zeta(s)=\beta(s)/s^{n-j+1}$ for $s>0$. The properties of $\beta$ and the definition of the class $T_j^n$ imply that $\zeta\in T_j^n$. We want to show that $\oz=\ot_{j,\zeta}^*$. By Theorem~\ref{thm:main_rep_maj} this is equivalent to
	\begin{equation}
		\label{eq:oz_maj_to_show}
		\oz(v)=\int_{\Rn} \xi(x) \d \MA_j(v;x)
	\end{equation} 
	for every $v\in\fconvs$, where $\xi(x)=\binom{n}{j}(\cR^{n-j} \zeta)(|x|) x$ for $x\in\Rn\setminus\{o\}$ and $\xi(o)=o$. In particular, Lemma~\ref{le:continuous_extension} and Lemma~\ref{le:r_trans_bij} show that $\xi\in C_c(\Rn;\Rn)$. Similarly, for $k\in\N$ let $\xi_k\in C_c(\Rn;\Rn)$ be given by $\xi_k(x)=\binom{n}{j} (\cR^{n-j} \zeta_k)(|x|) x$ for $x\in\Rn\setminus\{o\}$ and $\xi_k(o)=o$. Since $\beta_k$ converges uniformly to $\beta$ on $[0,\infty)$, it follows from
	\begin{align*}
		|\xi_k(x)-\xi(x)| &= \binom{n}{j} \left|(\beta_k(|x|)-\beta(|x|))\frac{x}{|x|} + (n-j)\left(\int_{|x|}^\infty \frac{\beta_k(t)-\beta(t)}{t^2}\d t \right) x \right|\\
		&\leq \binom{n}{j}\left( \|\beta_k-\beta\|_{\infty} + (n-j)\|\beta_k-\beta\|_{\infty} \int_{|x|}^\infty \frac{1}{t^2}\d t\; |x|\right)\\
		&= (1+ n-j) \binom{n}{j} \|\beta_k-\beta\|_{\infty}
	\end{align*}
	for $x\in\Rn\setminus\{o\}$ and $k\in\N$, that also $\xi_k$ converges uniformly to $\xi$ on $\Rn$ as $k\to\infty$. Moreover, $\supp(\xi_k),\supp(\xi)\subseteq R B^n$ for every $k\in\N$. Thus, using Theorem~\ref{thm:main_rep_maj} together with the dominated convergence theorem and the fact that $\MA_j(v;R B^n)<\infty$, we obtain
	\begin{align*}
		\oz(v)&=\lim_{k\to\infty} \ot_{j,\zeta_k}^*(v) =\lim_{k\to\infty} \int_{\Rn} \xi_k(x) \d\MA_j(v;x)= \int_{\Rn} \xi(x) \d\MA_j(v;x)
	\end{align*}
	for every $v\in\fconvf$, which shows \eqref{eq:oz_maj_to_show}.
	
	Lastly, if $\oz=\sum_{j=1}^n \ot_{j,\zeta_j}^*$ with $\zeta_j\in T_j^n$ for $j\in\{1,\ldots,n\}$, then by homogeneity together with Proposition~\ref{prop:retrieve_densities} and Lemma~\ref{le:retrieve_densities_n} we have
	\[
	\oz(\lambda\, w_s)=  \frac{\kappa_{n-1}}{n} \sum_{j=1}^{n} \binom{n}{j} \lambda^j s^{n-j+1} \zeta_j(s)  e_n
	\]
	for every $s,\lambda\geq 0$, which uniquely determines the densities $\zeta_j$.
	\qed	
	
	\subsection*{Acknowledgments}
	The authors wish to thank Jonas Knoerr, Martin Rubey, and Jacopo Ulivelli for helpful discussions and remarks. This project was supported by the Austrian Science Fund (FWF) 10.55776/P36210.
	

	\vfill
	
	\parbox[t]{8.5cm}{
		Mohamed A. Mouamine\\
		Institut f\"ur Diskrete Mathematik und Geometrie\\
		TU Wien\\
		Wiedner Hauptstra{\ss}e 8-10/104-06\\
		1040 Wien, Austria\\
		e-mail: mohamed.mouamine@tuwien.ac.at
	}
	
	\bigskip
	
	\parbox[t]{8.5cm}{
		Fabian Mussnig\\
		Institut f\"ur Diskrete Mathematik und Geometrie\\
		TU Wien\\
		Wiedner Hauptstra{\ss}e 8-10/104-06\\
		1040 Wien, Austria\\
		e-mail: fabian.mussnig@tuwien.ac.at
	}	
\end{document}